\documentclass[oneside,english]{amsart}
\usepackage[T1]{fontenc}
\usepackage[latin9]{inputenc}
\usepackage{amstext}
\usepackage{amsthm}
\usepackage{amssymb}
\usepackage{graphicx}

\makeatletter

\newcommand{\lyxdot}{.}

\numberwithin{equation}{section}
\numberwithin{figure}{section}
\theoremstyle{plain}
\newtheorem{thm}{\protect\theoremname}[section]
\theoremstyle{plain}
\newtheorem{conjecture}[thm]{\protect\conjecturename}
\theoremstyle{plain}
\newtheorem{question}[thm]{\protect\questionname}
\theoremstyle{plain}
\newtheorem{lem}[thm]{\protect\lemmaname}
\theoremstyle{definition}
\newtheorem{example}[thm]{\protect\examplename}
\theoremstyle{definition}
\newtheorem{defn}[thm]{\protect\definitionname}
\theoremstyle{plain}
\newtheorem{cor}[thm]{\protect\corollaryname}
\theoremstyle{remark}
\newtheorem*{acknowledgement*}{\protect\acknowledgementname}

\usepackage{hyperref}

\makeatother

\usepackage{babel}
\providecommand{\acknowledgementname}{Acknowledgement}
\providecommand{\conjecturename}{Conjecture}
\providecommand{\corollaryname}{Corollary}
\providecommand{\definitionname}{Definition}
\providecommand{\examplename}{Example}
\providecommand{\lemmaname}{Lemma}
\providecommand{\questionname}{Question}
\providecommand{\theoremname}{Theorem}

\begin{document}
\global\long\def\packf#1{\mathfrak{#1}}%

\global\long\def\vertexset{V}%

\global\long\def\chordlen{\textup{chord}}%

\global\long\def\edgeset{E}%

\global\long\def\circumradius{\boldsymbol{r}}%

\global\long\def\labelof{\lambda}%

\global\long\def\Rd{\mathbb{R}^{d}}%

\global\long\def\R{\mathbb{R}}%

\global\long\def\N{\mathbb{N}}%

\global\long\def\codes{\mathfrak{C}}%

\global\long\def\angles{\mathfrak{A}}%

\global\long\def\essential{\mathcal{E}}%

\global\long\def\parenth#1{\left(#1\right)}%

\global\long\def\radii{\textup{radii}}%

\global\long\def\centers{\textup{centers}}%

\global\long\def\codecenterlabel{\textup{center}\lambda}%

\global\long\def\codeneighborlabels{\lambda}%

\global\long\def\neighbors{\textup{neighbors}}%

\global\long\def\code{\textup{code}}%

\global\long\def\pcodes{\textup{codes}}%

\global\long\def\emb{\textup{emb}}%

\global\long\def\cal#1{\mathcal{#1}}%

\global\long\def\set#1#2{\left\{  \vphantom{{#1#2}}#1\right|\left.\vphantom{{#1#2}}#2\right\}  }%

\global\long\def\underleq#1{\underset{#1}{\blacktriangleleft}}%

\global\long\def\underleqtop#1{\underset{#1}{\vartriangleleft}}%

\newcommand{\eder}{Eder Kikianty}
\newcommand{\miek}{Miek Messerschmidt}

\newcommand{\ederemail}{eder.kikianty@up.ac.za}
\newcommand{\miekemail}{miek.messerschmidt@up.ac.za}

\newcommand{\upaddress}{
Department of Mathematics and Applied Mathematics, 
University of Pretoria, 
Private bag X20 Hatfield,
0028 Pretoria, 
South Africa 
}

\def\titletext{On compact packings of euclidean space\\ with spheres of finitely many sizes}
\def\titletextshort{On compact packings of euclidean space with spheres}

\newcommand{\abstracttext}{
For $d\in\mathbb{N}$, a compact sphere packing of Euclidean space
$\mathbb{R}^{d}$ is a set of spheres in $\mathbb{R}^{d}$ with disjoint
interiors so that the contact hypergraph of the packing is the vertex
scheme of a homogeneous simplicial $d$-complex that covers all of
$\mathbb{R}^{d}$.

We are motivated by the question: For $d,n\in\mathbb{N}$ with $d,n\geq2$,
how many configurations of numbers $0<r_{0}<r_{1}<\ldots<r_{n-1}=1$
can occur as the radii of spheres in a compact sphere packing of $\mathbb{R}^{d}$
wherein there occur exactly $n$ sizes of sphere?

We introduce what we call `heteroperturbative sets' of labeled triangulations
of unit spheres and we discuss the existence of non-trivial examples
of heteroperturbative sets. For a fixed heteroperturbative set, we
discuss how a compact sphere packing may be associated to the heteroperturbative
set or not.

We proceed to show, for $d,n\in\mathbb{N}$ with $d,n\geq2$ and for
a fixed heteroperturbative set, that the collection of all configurations
of $n$ distinct positive numbers that can occur as the radii of spheres
in a compact packing is finite, when taken over all compact sphere
packings of $\mathbb{R}^{d}$ which have exactly $n$ sizes of sphere
and which are associated to the fixed heteroperturbative set.}

\newcommand{\subjectcodesprimary}{%
    52C17, 
    05B40
}

\newcommand{\subjectcodessecondary}{%
    57Q05, 
    52C26
}

\newcommand{\keywordstext}{sphere packing, compact sphere packing}

\title[On compact packings of euclidean space with spheres]{On compact packings of euclidean space\\ with spheres of finitely many sizes}

\author{\eder}
\address{\eder,\ \upaddress}
\email{\ederemail}

\author{\miek}
\address{\miek,\ \upaddress}
\email{\miekemail}

\begin{abstract}
\abstracttext
\end{abstract}


\subjclass[2010]{Primary: \subjectcodesprimary. Secondary: \subjectcodessecondary}

\maketitle

\section{Introduction}

Sphere packing in general has many connections to other branches of
mathematics, see for example \cite{ConwaySloane}. Consider the disc
packing in Figure~\ref{fig:pack54}. Such a packing is called `compact\emph{'
}(we give a precise definition later)\emph{.} Apart from being objects
of purely mathematical interest, some compact packings have been observed
to arise in self-assembled nano-structures: Compare the nine compact
disc packings packings from \cite{Kennedy2006} with electron micrographs
from the non-mathematical literature \cite{ChemistryFernique}, \cite{ChemistryPaik}
and \cite{ChemistryCherniukh}. Compact sphere and disc packings are
therefore of special interest, not just within mathematics, but also
in other scientific disciplines.

\medskip

In this paper we are motivated by the following question: For $d,n\in\mathbb{N}$
with $d,n\geq2$, how many configurations of numbers $0<r_{0}<r_{1}<\ldots<r_{n-1}=1$
can occur as the radii of spheres in a compact sphere packing of $\mathbb{R}^{d}$
wherein there occur exactly $n$ sizes of sphere?

\medskip

\begin{figure}
    \centering \includegraphics[width=100mm]{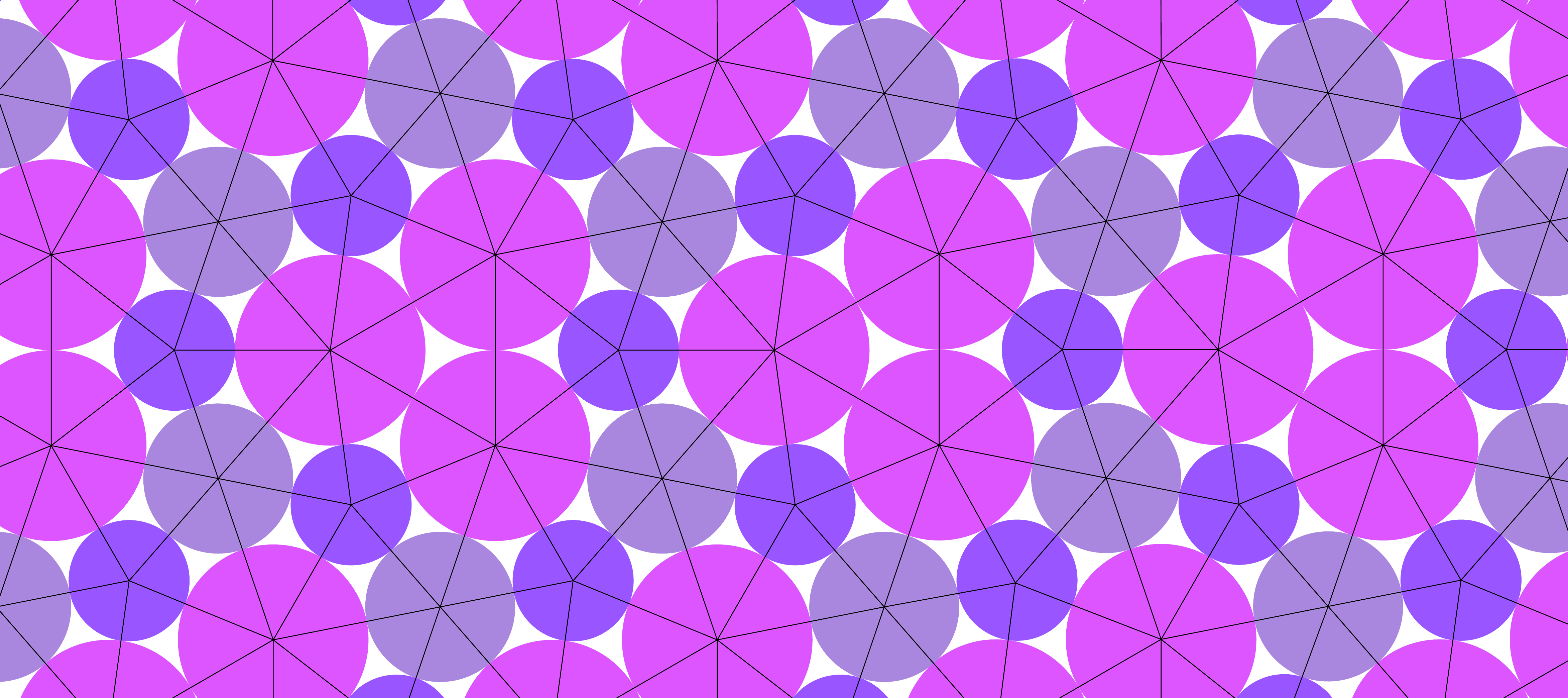}\caption{\label{fig:pack54}One of the 164 compact packings of $\protect\R^{2}$
        by discs with three sizes from \cite[Packing 54]{Fernique2021}, here
        with its packing complex overlaid. }
\end{figure}

Let $d\in\N$. By a \emph{sphere packing} $\packf p$ of Euclidean
space $\mathbb{R}^{d}$, we mean a family of spheres with pairwise
disjoint interiors in $\Rd$. By $\radii(\packf{p\packf )}$ we denote
the set of all positive real numbers that occur as radii of spheres
from $\packf p$.

The \emph{contact hypergraph }of a sphere packing $\packf p$ of $\mathbb{R}^{d}$
is defined as follows: The vertex set is taken as the set of all the
centers of spheres in the packing $\packf p$. As hyperedges, we take
exactly the subsets $E$ of the vertex set that have cardinality at
most $d+1$ and that satisfy both of the following two conditions:
Firstly, for every pair of distinct vertices $a$ and $b$ from the
set $E$, the corresponding spheres in $\packf p$ that respectively
have $a$ and $b$ as centers are mutually tangent. Secondly, if any
vertex of the hypergraph is an element of the convex hull of the set
$E$, then this vertex is an element of $E$.

We say a sphere packing $\packf p$ in Euclidean space $\mathbb{R}^{d}$
is \emph{compact }if the contact hypergraph is exactly the vertex
scheme of a homogeneous\footnote{Here `homogeneous' means that every simplex of the complex is contained
    in a $d$-simplex that is also in the complex.} simplicial $d$-complex in $\Rd$ whose underlying space is homeomorphic
to $\Rd$ through the identity map. We call this homogeneous simplicial
$d$-complex the \emph{packing complex }of the compact sphere packing
$\packf p$. Intuitively speaking, $\Rd$ is tesselated by the $d$-simplices
in the packing complex (cf. Figure~\ref{fig:pack54}).

In this paper we restrict our view to compact sphere packings $\packf p$
satisfying $|\radii(\packf p)|<\infty$ and $\max\radii(\packf p)=1$.
With $d,n\in\N$ we define
\[
    \Pi_{d,n}:=\set{\radii(\packf p)}{\begin{array}{c}
            \packf p~\text{a compact sphere packing of }\Rd \\
            |\radii(\packf p)|=n,\ \max\radii(\packf p)=1
        \end{array}}.
\]
Our work in this paper is motivated by the following conjecture.
\begin{conjecture}
    \label{conj:main-conjecture}For all $d,n\in\N$ with $d,n\geq2$,
    the set $\Pi_{d,n}$ is finite.
\end{conjecture}

Although this conjecture remains open in full generality, some certain
special cases have been resolved and we discuss them shortly.

In this paper we prove a general result that can be used to show,
for all $d,n\in\N$ with $d,n\geq2$, that certain subsets of $\Pi_{d,n}$
are finite. To this end, in Section~\ref{subsec:spherical-triangulations-and-heteroperturbative-sets},
we define what we call heteroperturbative sets of labeled triangulations
of unit spheres (of any dimension). Roughly, what this means is that,
if a spherical triangulation from a heteroperturbative set is perturbed
while preserving its combinatorics, then one edge must grow in length,
and one edge must shrink in length (cf. Section~\ref{subsec:spherical-triangulations-and-heteroperturbative-sets}).
Every sphere in a compact sphere packing in $\Rd$ has canonically
associated to it a labeled triangulation of the unit sphere in $\Rd$
(cf. Section~\ref{subsec:Canonical-labeling-and}). For a heteroperturbative
set $\mathbf{S}$ of labeled triangulations of unit spheres and a
compact sphere packing $\packf p$ in $\Rd$, we say that \emph{$\packf p$
    is associated to $\mathbf{S}$ }if every sphere from $\packf p$ has
its canonical associated labeled triangulation being an element of
$\mathbf{S}$.

Our main result in this paper is:
\begin{thm}
    \label{thm:main-theorem}For all $d,n\in\N$ with $d,n\geq2$ and
    $\mathbf{S}$ any heteroperturbative set of labeled triangulations
    of unit spheres, the set
    \[
        \Pi_{d,n}(\mathbf{S}):=\set{\radii(\packf p)}{\begin{array}{c}
                \packf p~\text{a compact sphere packing in }\Rd,           \\
                \packf p\ \text{is associated to \ensuremath{\mathbf{S}},} \\
                |\radii(\packf p)|=n,\ \max\radii(\packf p)=1
            \end{array}}
    \]
    is finite.
\end{thm}

It is easily seen that the set of \emph{all} triangulations of the
unit circle in $\R^{2}$ is heteroperturbative (cf. Example~\ref{exa:R2-heteroperturbative}),
so the special case for $d=2$ of Conjecture~\ref{conj:main-conjecture}
is proven by applying Theorem \ref{thm:main-theorem} (this is the
main result of \cite{Messerschmidt2d}). However, the set all triangulations
of the unit sphere in $\R^{3}$ is not heteroperturbative (cf. Example~\ref{exa:R3-not-heteroperturbative}).
Therefore, to be able to apply Theorem~\ref{thm:main-theorem} in
attacking Conjecture~\ref{conj:main-conjecture} for $d=3$, one
is required to find an appropriate proper subset of the set of all
labelled triangulations of the unit sphere in $\R^{3}$, and additionally,
prove that it is heteroperturbative.

In general it is difficult to determine if a set of triangulations
of unit spheres is heteroperturbative. A consequence of recent work
by Winter \cite[Corollary~4.13]{WinterPolytopes} shows that the set,
that we denote $\mathbf{W}$ (cf. Example~\ref{exa:Winter-polytopes-heteroperturbative}),
of all triangulations of unit spheres of any dimension that are determined
from inscribed convex polytopes that contain the center of their circumsphere
in the interior, is heteroperturbative. Hence, non-trivial heteroperturbative
sets do exist, and by extension, Theorem~\ref{thm:main-theorem}
shows that non-trivial and fairly interesting subsets of $\Pi_{d,n}$
are finite.

We also exhibit a set of labeled spherical triangulations, denoted
$\mathbf{Q}$, that is such that $\Pi_{d,n}=\Pi_{d,n}(\mathbf{Q})$
for all $d,n\geq2$ (cf. Section~\ref{subsec:spherical-triangulations-and-heteroperturbative-sets}).
Clearly, a proof that $\mathbf{Q}$ is heteroperturbative would, through
application of Theorem~\ref{thm:main-theorem}, immediately prove
Conjecture~\ref{conj:main-conjecture}. However, whether or not $\mathbf{Q}$
is heteroperturbative is an open problem to the best of our knowledge.

$\smallskip$

Before continuing, we briefly remark on some relevant history regarding
the family of sets $\Pi_{d,n}$ for $d,n\in\N$.

It is easily seen that $|\Pi_{1,n}|=\infty$ for all $n\geq2$ and
that $|\Pi_{2,1}|=1$. For dimensions $d\geq3$, regular simplices
do not tesselate $\Rd$, because $2\pi$ is not an integer multiple
of the dihedral angles occurring in a regular $d$-simplex \cite[Remark~2]{ChoDihedralAngles}.
Hence $|\Pi_{d,1}|=0$ for all dimensions $d\geq3$. In particular,
the famous $E_{8}$ and Leech lattice sphere packings (cf. \cite{ConwaySloane})
are \emph{not }compact packings.

For dimension $2$, in 2006, Kennedy showed that $|\Pi_{2,2}|=9$
by computing all elements of the set $\Pi_{2,2}$ \cite{Kennedy2006}.
Determining $|\Pi_{2,3}|$ turned out to be more challenging and took
a number of years. A finite bound for $|\Pi_{2,3}|$ was determined
by the second named author in \cite{Messerschmidt2020}, and working
roughly at the same time $\Pi_{2,3}$ was computed and was shown to
have $164$ elements by Fernique, Hashemi, and Sizova in \cite{Fernique2021}.
A crucial ingredient in showing that $|\Pi_{2,3}|$ is finite is the
following result:
\begin{lem}
    \cite[Lemma~6.1]{Fernique2021}\label{lem:FerniqueBabyBootstrap}
    There exists a constant $K>0$ so that every compact disc packing
    $\packf p$ of $\R^{2}$ with $\radii(\packf p)=\{r_{0},r_{1},r_{2}\}$
    satisfying $0<r_{0}<r_{1}<r_{2}=1$ is such that
    \[
        K\leq\frac{r_{0}}{r_{1}}.
    \]
\end{lem}

We specifically make note of the order of the quantifiers in the above
result. The constant $K$ is independent of the choice of compact
packing with discs of three sizes and can be taken to be $\min\{r_{0}:\{r_{0},1\}\in\Pi_{2,2}\}$.
That the ratio of the small and medium size discs in a compact packing
with three sizes of discs is related to the radii that occur in compact
packings with two sizes of discs is indicative that an inductive argument
may be used to glean information on the sets $\Pi_{2,n}$ for $n>3$.
Exploiting this idea in \cite{Messerschmidt2d}, the second named
author showed that $|\Pi_{2,n}|$ is finite for all $n\geq2$. Although
the cardinality of these sets could be shown to be finite, the methods
in \cite{Messerschmidt2d} are non-constructive and do not produce
quantitative bounds on the cardinality of the sets $\Pi_{2,n}$ for
$n\geq2$.

For dimension $3$, in \cite{FerniqueTwoSpheres2021,FerniqueThreeSizes},
Fernique computed $\Pi_{3,2}$ and $\Pi_{3,3}$ which respectively
were shown to have cardinality $1$ and $4$. Fernique himself calls
his results disappointing for the reason that, in three dimensions,
compact packings of spheres of two or three sizes arise only through
the filling of interstitial holes of close-packings of unit spheres.
This is in contrast to the case in two dimensions where, for example,
there exists a rich variety of compact disc packings with three sizes
of discs that do not arise from compact disc packings with two sizes
of discs through merely filling in interstitial holes with discs (cf.
\cite{Fernique2021}). This contrast in richness raises the following
question:
\begin{question}
    Does there exist a dimension greater than two that admits a compact
    packing of spheres with two or more sizes which does not arise through
    filling interstitial holes in a lattice packing of unit spheres?
\end{question}

\smallskip

Throughout this paper we make use of standard terminology regarding
simplicial complexes and related concepts. The reader unfamiliar with
the terminology may find definitions in \cite{Munkres}. The specific
standard terms we make use of are: \emph{simplicial complex}, $k$\emph{-skeleton}
of a simplicial complex, \emph{abstract simplicial complex}, the \emph{underlying
    (topological) space} of a simplicial complex, the \emph{vertex scheme}
of a simplicial complex (i.e., the abstract simplicial complex that
results from forgetting all geometric information of the simplicial
complex), \emph{isomorphism of abstract simplicial complexes}, the
\emph{closed star} of a vertex in a simplicial complex, the \emph{link}
of a vertex in a simplicial complex.\medskip

We now describe the path toward proving Theorem~\ref{thm:main-theorem}.
The main flavor follows that of \cite{Messerschmidt2d}. Before giving
more detail, we briefly describe the idea, which consists of three
main parts. For a fixed dimension $d$, firstly we define an abstract
discrete structure that is shown to always occur in every compact
packing (cf. Section~\ref{sec:Fundamental-sets-of}). Secondly, we
show that if such an abstract discrete structure arises from a compact
packing associated to some heteroperturbative set of labeled triangulations
of unit spheres, then this structure uniquely determines the radii
of spheres in the compact packing (cf. Section~\ref{sec:Uniqueness}).
Thirdly, we show that if such an abstract discrete structure arises
from a compact packing with $n$ sizes of spheres that is associated
to some fixed heteroperturbative set, then there exists a bound on
the size of the structure and, more importantly, this bound can be
chosen so as to depend only on the number of sizes of spheres $n$
and the dimension $d$, and independently of any particular packing
(cf. Section~\ref{sec:Essential-sets}). We note that, although these
bounds are shown to exist, the argument to show their existence is
non-constructive. We conclude that only finitely many such abstract
discrete structures can arise from compact packings associated to
a heteroperturbative set, and the ones which do arise in this way
uniquely determine the radii in a compact packing associated to some
heteroperturbative set. This establishes Theorem~\ref{thm:main-theorem}.\smallskip

We now move to describing the content of this paper in more detail.

Preliminary definitions and results are described in a number of subsections
of Section~\ref{sec:Preliminaries}. In Section~\ref{subsec:Packing-codes},
we introduce what we call \emph{packing codes}. These are abstract
simplicial complexes with labeled vertices that will be used to abstractly
describe a sphere and its neighbors in a compact packing. We introduce
\emph{angle symbols}, \emph{realizations,} and \emph{realizers} in
Section~\ref{subsec:Angle-symbols-and}. Angle symbols are used to
abstractly represent angles in the triangle formed by connecting the
centers of three mutually tangent spheres with disjoint interiors
and indeterminate radii. A realization by a realizer assigns specific
values to indeterminates in formal arithmetic expressions. Typically
realized expressions will be expressions involving angle symbols derived
from packing codes. An important result is Lemma~\ref{lem:realize_angle_symbol_strict_monotone},
which establishes a monotone relationship between realizers and the
values of realized angle symbols. This result is used later, as a
first ingredient, in proving both Theorem~\ref{thm:canonical-realizers-are-unique-for-packings}
and Lemma~\ref{lem:bootstrapping-lemma2-1}, which we are yet to
discuss. In Section~\ref{subsec:spherical-triangulations}, we define
what we mean by a \emph{spherical simplicial complex} and by a \emph{triangulation
    of the unit sphere in $\Rd$.} We prove an easy lemma (Lemma~\ref{lem:spherical-triangulations-contain-center-in-interior})
regarding the geometry of spherical triangulations of unit spheres.
Although easy, this result is nevertheless crucial to be able to apply
Winter's results from \cite{WinterPolytopes} in the subsequent section.
Section~\ref{subsec:spherical-triangulations-and-heteroperturbative-sets}
sees definition of what we call \emph{heteroperturbative sets (of
    labeled spherical triangulations of unit spheres)} and we present
a few relevant examples, non-examples, and conjectured heteropertubative
sets. In particular, we show that the sets $\mathbf{T}_{2}$ and $\mathbf{W}$
(Examples~\ref{exa:R2-heteroperturbative} and ~\ref{exa:Winter-polytopes-heteroperturbative})
are heteroperturbative, that $\mathbf{T}_{3}$ (Example~\ref{exa:R3-not-heteroperturbative})
is not heteroperturbative, and we conjecture that $\mathbf{Q}$ (Definition~\ref{def:define-Q})
is heteroperturbative. In Section~\ref{subsec:Canonical-labeling-and},
we discuss how a compact packing $\packf p$ of spheres in $\Rd$
is canonically labeled, determines a canonical realizer, and how each
sphere in $\packf p$ canonically determines a labeled spherical triangulation
of the unit sphere $\Rd$, and by further forgetting the geometric
structure of the labeled triangulation, the sphere canonically determines
a packing code associated to the sphere.

In Section~\ref{sec:Fundamental-sets-of}, we define what we call
a \emph{fundamental set of packing codes}. We show that every compact
packing gives rise to such a fundamental set of packing codes (cf.
Theorem~\ref{thm:packings-determine-fundamental-set} and Figure~\ref{fig:pack-five-labelled}).

The main result in Section~\ref{sec:Uniqueness} is Theorem~\ref{thm:canonical-realizers-are-unique-for-packings},
showing that the radii of spheres in a compact packing associated
a heteroperturbative set is uniquely determined by any fundamental
set of packing codes that arises from the packing. As mentioned, Lemma~\ref{lem:realize_angle_symbol_strict_monotone}
is used together the defining property of a heteroperturbative set
to prove Lemma~\ref{lem:uniqueness-lemma} which shows how the existence
of labeled spherical triangulations in a heteroperturbative set that
are combinatorially equivalent to packing codes from a fundamental
set and whose edge lengths are determined through a realizer, actually
uniquely determines the realizer. Lemma~\ref{lem:packing-determines-triangulations-from-Q}
shows that labeled spherical triangulations as required in Lemma~\ref{lem:uniqueness-lemma}
always exist in every compact packing.

The Bootstrapping Lemma (Lemma~\ref{lem:bootstrapping-lemma2-1})
is proven in Section~\ref{sec:The-bootstrapping-lemma}. This lemma
is a generalization of Lemma~\ref{lem:FerniqueBabyBootstrap} and
is what makes the induction step work in a strong induction toward
proving the main technical result of the paper, Lemma~\ref{lem:essential-sets-finite-induction-step},
from whence the name. The Bootstrapping Lemma relates the ratios of
some of the values of two realizers under certain assumptions of the
existence of labeled spherical triangulations of unit spheres from
a heteroperturbative set that are combinatorially equivalent to a
packing code from fundamental set. Again as mentioned, Lemma~\ref{lem:realize_angle_symbol_strict_monotone}
along with the defining property of a heteroperturbative set is used
in proving The Bootstrapping Lemma.

Section~\ref{sec:Essential-sets} sees the proof of the main result
of this paper, Theorem~\ref{thm:main-theorem}. For $d,n\in\N$ with
$d,n\geq2$ and a heteroperturbative set $\mathbf{S}$, we introduce
what we call an \emph{$n$-essential set} \emph{for }$\mathbf{S}$\emph{
    in dimension }$d$, which adds further technical conditions on a fundamental
set, related to the existence of two realizers $\rho$ and $\sigma$
and of labeled triangulations from $\mathbf{S}$ relating the two
realizers $\rho$ and $\sigma$ to each other. We observe that these
further conditions are automatically satisfied by fundamental sets
that arise from an actual compact packing in $\Rd$ associated to
$\mathbf{S}$ with $n$ sizes of spheres: The canonical labeled triangulations
associated to spheres in the packing are taken, and for $\sigma$
and $\rho$ we take both equal to the canonical realizer for the packing.
This observation, along with the previous observation that radii of
the spheres in the packing being uniquely determined by the fundamental
set (Lemma~\ref{lem:uniqueness-lemma}), thus shows that the cardinality
of $\Pi_{d,n}(\mathbf{S})$ is at most cardinality of the set, denoted
$\essential_{d,n}(\mathbf{S})$, of all $n$-essential sets for $\mathbf{S}$
in dimension $d$, (cf. Lemma~\ref{lem:Pidn-cardinality-bounded-by-n-essential-sets}).
For any fixed $d$, we thus proceed to show that the set $\essential_{d,n}(\mathbf{S})$
is finite through strong induction on the number $n$. As the base
step of a strong induction, we show that the set $\essential_{d,2}(\mathbf{S})$
is finite (cf. Lemma~\ref{lem:2-ess-finite}). Lemma~\ref{lem:essential-sets-finite-induction-step},
the strong induction step, is the main technical result. Given that
the sets $\essential_{d,k}(\mathbf{S})$ are finite for all $k\in\{2,\ldots,n-1\}$,
using The Bootstrapping Lemma (Lemma~\ref{lem:bootstrapping-lemma2-1}),
we show that the set $\essential_{d,n}(\mathbf{S})$ is finite by
bounding the size of any element of $\essential_{d,n}(\mathbf{S})$
using only features of the finitely many elements in $\bigcup_{k=2}^{n-1}\essential_{d,k}(\mathbf{S})$.
This proves the main result of this paper in Corollary~\ref{cor:main-result}.

\section{Preliminaries\label{sec:Preliminaries}}

\subsection{Packing codes\label{subsec:Packing-codes}}

Let $d\in\N$ with $d\geq2$. Let $\Sigma$ be any set of symbols.
By a \emph{$d$-dimensional} \emph{packing code over $\Sigma$}, or
just a \emph{packing code}, we mean a symbol $c:T$, with $c\in\Sigma$
and $T$ an abstract homogeneous simplicial $(d-1)$-complex with
vertices labeled by elements from $\Sigma$. In a packing code $c:T$,
we call $c$ the \emph{center }and $T$ the \emph{neighbor complex}
of the packing code. For a vertex $v$ in $T$ we denote its label
by $\labelof v\in\Sigma$. We define $\codeneighborlabels(T):=\{\labelof v:v\text{ a vertex in }T\}$.
By $\codes(\Sigma)$, we denote the set of all $d$-dimensional packing
codes. The dependence on the dimension $d$ is suppressed in the notation
$\codes(\Sigma)$, but this should not cause confusion as $d$ is
fixed at all times.

\subsection{Angle symbols and realizations\label{subsec:Angle-symbols-and}}

Let $\Sigma$ be any set of symbols. For symbols $a,b,c\in\Sigma$
we denote the formal symbol
\[
    c_{b}^{a}:=\arccos\parenth{\frac{(c+a)^{2}+(c+b)^{2}-(a+b)^{2}}{2(c+a)(c+b)}}.
\]
We call $c_{b}^{a}$ an \emph{angle symbol} (over $\Sigma$). We \emph{always}
regard elements from $\Sigma$ in an angle symbol as indeterminates.
With three mutually tangent spheres with pairwise disjoint interiors,
with centers labeled $a,b$, and $c$ forming a triangle, the symbol
$c_{b}^{a}$ is meant to abstractly represent the magnitude of the
angle formed at $c$ (cf. Figure~\ref{fig:angle-symbol}). In the
angle symbol $c_{b}^{a}$, we call $c$ the \emph{vertex} of the symbol
$c_{b}^{a}.$

We define a \emph{realizer }to be a map $\rho:\Sigma\to(0,\infty)$.
Given a realizer $\rho:\Sigma\to(0,\infty)$ and a formal arithmetic
expression $E$, with symbols from $\Sigma$, we denote by $E|_{\rho}$
the expression $E$ with every symbol $s\in\Sigma$ occurring in $E$
replaced by $\rho(s).$ We call $E|_{\rho}$ \emph{the realization
    of $E$ by $\rho$.}

An important result that we use multiple times in this paper is the
following lemma that establishes a specific monotone relationship
between realizers and realized angle symbols when the values of the
realizer is kept constant on the vertex of the angle symbol.

\begin{figure}
    \centering \includegraphics[width=30mm]{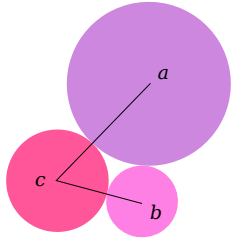}\caption{\label{fig:angle-symbol}The angle represented by the angle symbol
        $c_{b}^{a}.$}
\end{figure}

\begin{lem}
    \label{lem:realize_angle_symbol_strict_monotone}Let $\Sigma$ be
    any finite set of symbols and let $a,b,c\in\Sigma$ with $c\neq a$.
    Let $\rho:\Sigma\to(0,\infty)$ and let $\nu:\Sigma\to[0,\infty)$
    be a non-zero function.
    \begin{enumerate}
        \item The map $(0,\infty)\ni t\mapsto c_{b}^{a}|_{t\rho}$ is constant.
        \item If $\nu(c)=0$ and $\nu(a)>0$, then the map $[0,\infty)\ni t\mapsto c_{b}^{a}|_{\rho+t\nu}$
              is strictly increasing.
        \item If $\nu(c)=\nu(a)=\nu(b)=0$, then the map $[0,\infty)\ni t\mapsto c_{b}^{a}|_{\rho+t\nu}$
              is constant.
    \end{enumerate}
\end{lem}

\begin{proof}
    With distinct symbols $a,b,c,d\in\Sigma$, taking partial derivatives
    of the expression $c_{a}^{b}$ toward the indicated symbol, we obtain
    the following
    \[
        \begin{array}{ccc}
            \begin{array}{rcl}
                \partial_{a}c_{a}^{b}=\partial_{a}c_{b}^{a} & = & \frac{\sqrt{bc}}{(c+a)\sqrt{a}\sqrt{a+b+c}},\vspace{2mm}                     \\
                \partial_{a}c_{a}^{c}=\partial_{a}c_{c}^{a} & = & \frac{c}{(c+a)\sqrt{a^{2}+2ac}},\vspace{2mm}                                 \\
                \partial_{a}c_{a}^{a}                       & = & \frac{2\sqrt{c}}{(c+a)\sqrt{2a+c}},\vspace{4mm}                              \\
                \partial_{c}c_{a}^{b}=\partial_{c}c_{b}^{a} & = & -\frac{(a+b+2c)\sqrt{ab}}{(c^{2}+ab+ac+bc)\sqrt{c}\sqrt{a+b+c}},\vspace{2mm} \\
                \partial_{c}c_{a}^{c}=\partial_{c}c_{c}^{a} & = & -\frac{a}{(c+a)\sqrt{a^{2}+2ac}},
            \end{array} & \quad & \begin{array}{rcl}
                \partial_{d}c_{b}^{a} & = & 0, \\
                \partial_{d}c_{c}^{a} & = & 0, \\
                \partial_{d}c_{a}^{a} & = & 0, \\
                \partial_{d}c_{c}^{c} & = & 0, \\
                \partial_{c}c_{c}^{c} & = & 0.
            \end{array}\end{array}
    \]

    It is easily seen from the definition of the expression $c_{b}^{a}$
    that the map $(0,\infty)\ni t\mapsto c_{b}^{a}|_{t\rho}$ is constant.

    Let $t\in[0,\infty)$ be arbitrary. Assuming $\nu(c)=0$, then the
    directional derivative of map $(0,\infty)^{S}\ni\sigma\mapsto\alpha(x)|_{\sigma}$
    at $\rho+t\nu$ in the direction of $\nu$ is a positive scalar multiple
    of
    \begin{align*}
        \sum_{s\in\Sigma}(\partial_{s}c_{b}^{a})|_{\rho+t\nu}\ \nu(s) & =\text{\ensuremath{\sum_{s\in\Sigma\backslash\{c\}}}(\ensuremath{\partial_{s}c_{b}^{a}})\ensuremath{|_{\rho+t\nu}\ \nu}(s)}\geq0.
    \end{align*}
    Assuming further that $\nu(a)>0$, then at least one term of the above
    summation is non-zero, and hence the map $[0,\infty)\ni t\mapsto c_{b}^{a}|_{\rho+t\nu}$
    is strictly increasing. On the other hand, assuming $\nu(c)=\nu(a)=\nu(b)=0$
    then every term in the above summation is zero, and hence the map
    $[0,\infty)\ni t\mapsto c_{b}^{a}|_{\rho+t\nu}$ is constant.
\end{proof}

\subsection{Spherical triangulations of unit spheres\label{subsec:spherical-triangulations}}

Let $d\geq2$, $1\leq k\leq d$. Let $G\subseteq\Rd$ be a set of
$k$ points with $0\notin G$ so that so that $G\cup\{0\}$ is affinely
independent. By the \emph{spherical $(k-1)$-simplex }(defined by
$G$) we mean the central projection of the $(k-1)$-simplex defined
by $G$ to the unit sphere of $\Rd$ centered at $0$. Since $0\notin G$,
the spherical $(k-1)$-simplex is homeomorphic to the simplex defined
by $G$ through central projection, and all faces of the spherical
$(k-1)$-simplex are contained in great spheres of the unit sphere
of $\Rd$ of appropriate dimension.

By a \emph{spherical simplicial complex}, we mean a simplicial complex,
except with all simplices in the complex being spherical simplices
rather than usual simplices. By a \emph{spherical triangulation }of
the unit sphere in $\Rd$, we mean a homogeneous spherical simplicial
$(d-1)$-complex so that the underlying space of the spherical simplicial
complex is homeomorphic to the unit sphere through the identity map,
i.e., the spherical $(d-1)$-simplices of the spherical simplicial
complex tesselate the unit sphere of $\Rd$.

We will say a spherical triangulation is\emph{ combinatorially equivalent
}to an abstract simplicial complex $A$ if the vertex scheme of the
spherical triangulation is isomorphic to the abstract simplicial complex
$A$. By overloading the term, we will say two spherical triangulations
are \emph{combinatorially equivalent }if the vertex schemes of both
triangulations are isomorphic to the same abstract simplicial complex.

We will say that two combinatorially equivalent spherical triangulations
$P$ and $Q$ are \emph{edge-isometric} if, with respect to the geodesic
metric on the unit sphere, every edge of $P$ is equally long to the
corresponding edge of $Q$.

Before proceeding, we prove an elementary result regarding spherical
triangulations.
\begin{lem}
    \label{lem:spherical-triangulations-contain-center-in-interior}Let
    $d\geq2$ and let $P$ be a spherical triangulation of the unit sphere
    of $\Rd$.
    \begin{enumerate}
        \item There exists no hyperplane $H\subseteq\Rd$ (through the origin) so
              that all vertices of $P$ lie in one of the two closed half-spaces
              determined by $H$.\smallskip
        \item The interior of the closed convex hull of all vertices from $P$ is
              non-empty.\smallskip
        \item The center of the unit sphere, i.e. the origin, is contained in the
              interior of the closed convex hull of all vertices from $P$.
    \end{enumerate}
\end{lem}

\begin{proof}
    We prove (1). Suppose $H$ is a hyperplane so that all vertices of
    $P$ lie in one of the two closed half-spaces determined by $H$.
    By our definition of spherical simplices through central projection,
    all simplices in $\Rd$ that define the spherical simplices of $P$
    lie in this same closed half-space. There thus exists a point on the
    unit sphere in the opposite open half-space of $H$, that is not covered
    by $P$. We conclude that no such hyperplane $H$ exists.

    We prove (2). Suppose that the closed convex hull $C$ of all vertices
    from $P$ has empty interior. Then $C$ is contained in some affine
    hyperplane $H$ of $\Rd$. Hence $C$ lies inside, or to one side
    of the hyperplane through the origin parallel to $H$, contradicting
    (1).

    We prove (3). Suppose that the closed convex hull $C$ of all vertices
    of $P$ does not contain the center of the unit sphere as an interior
    point. There are two possibilities, either the center is a boundary
    point of $C$, or is an exterior point of $C$. By (2), $C$ has non-empty
    interior. In both cases, using either the Supporting Hyperplane Theorem
    or the Separation Theorem (see for instance \cite[Theorems~2 and~3, p.133]{LuenbergerOptimization}),
    there exists a hyperplane so that all vertices of $P$ lie in one
    of the two closed half-spaces defined by this hyperplane, contradicting
    (1).
\end{proof}

\subsection{Labeled spherical triangulations and heteroperturbative sets\label{subsec:spherical-triangulations-and-heteroperturbative-sets}}

By $\mathbf{T}$ we denote the set of all labeled spherical triangulations
of unit spheres of any finite dimension with the center of the unit
sphere (the origin) and all vertices of the triangulation carrying
labels from $\N\cup\{0\}$. For $P\in\mathbf{T}$ and vertex $v$
in $P$, we define the notation $\labelof v$ to denote the label
in $\N\cup\{0\}$ of the vertex $v$ in $P$.

Let $P,Q\in\mathbf{T}$. We say $P$ and $Q$ are \emph{combinatorially
    equivalent with identical labels }if $P$ and $Q$ are combinatorially
equivalent and all pairs the corresponding vertices of $P$ and $Q$
and their centers carry identical labels. This notion is an equivalence
relation on $\mathbf{T}$ and we denote the equivalence class of $P$
in $\mathbf{T}$ by $[P]$.

We introduce the term `heteroperturbative' subset of $\mathbf{T}$,
which we define formally in the next paragraph but first describe
intuitively to hopefully aid the reader's understanding. Informally,
a heteroperturbative subset of $\mathbf{T}$ is such that if one perturbs
a labeled spherical triangulation from the heteroperturbative set,
(while retaining its combinatorial structure and also remaining inside
the heteroperturbative set), there must exist an edge of the triangulation
that grows in length and another edge that shrinks in length.

Formally, Let $\mathbf{S}\subseteq\mathbf{T}$ and let $R\in\mathbf{S}$.
For $P\in[R]\cap\mathbf{S}$ and \emph{$Q\in\mathbf{S}$}, we say
that $Q$ is a \emph{perturbation of $P$ }if $Q\in[R]\cap\mathbf{S}$
and $P$ is not edge-isometric to $Q$. We say the equivalence class
$[R]\cap\mathbf{S}$ in $\mathbf{S}$ is \emph{heteroperturbative,}
if the following holds for any pair $P,Q\in[R]$: If $P$ and $Q$
are not edge-isometric, then there exist two pairs of respectively
corresponding edges, $uv$ and $xy$ of $P$, and $u'v'$ and $x'y'$
of $Q$, so that, with respect to the geodesic metric on the unit
sphere, the edge $uv$ in $P$ is strictly longer than the corresponding
edge $u'v'$ in $Q$; and, on the other hand, the edge $xy$ in $P$
is strictly shorter than the corresponding edge $x'y'$ in $Q$. We
say that $\mathbf{S}$ is a \emph{heteroperturbative set} (of labeled
spherical triangulations of unit spheres) if all equivalence classes
in $\mathbf{S}$ are heteroperturbative.

\medskip

Whether non-trivial/sufficiently interesting heteroperturbative sets
even exist is an obvious question, which we now briefly discuss.

It is an easy exercise to see that the set of all triangulations of
the unit circle in $\R^{2}$ is heteroperturbative:
\begin{example}
    \label{exa:R2-heteroperturbative}Define $\mathbf{T}_{2}\subseteq\mathbf{T}$
    as all labeled spherical triangulations of the unit circle in $\R^{2}$.
    The geodesic lengths of all edges in a triangulation from $\mathbf{T}_{2}$
    sum to $2\pi$. Hence, for any $P\in\mathbf{T}_{2}$, we cannot have
    \emph{all }edges in a perturbation of $P$ grow (shrink) as this would
    mean that the sum of the edge lengths in the perturbation sum to strictly
    more (less) than $2\pi$. We note that here the labels on vertices
    are superfluous.
\end{example}

In contrast to the case for $\R^{2},$ the set of all labeled triangulations
of the unit sphere in $\R^{3}$ is not heteroperturbative, as the
following example demonstrates:
\begin{example}
    \label{exa:R3-not-heteroperturbative}Define $\mathbf{T}_{3}\subseteq\mathbf{T}$
    as all labeled spherical triangulations of the unit sphere in $\R^{3}$.
    Triangulating the unit sphere in $\R^{3}$ with an equatorial strip
    of bisected darts (cf. Figure~\ref{fig:equatorial-chevrons}) shows
    that $\mathbf{T}_{3}$ is not heteroperturbative.
\end{example}

\begin{figure}[t]
    \centering \includegraphics[width=0.7\textwidth]{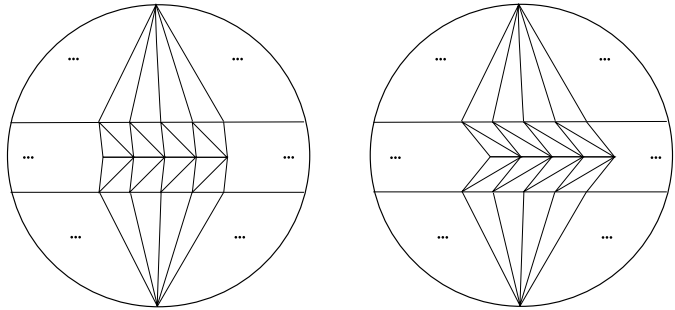}

    \caption{\label{fig:equatorial-chevrons}The set $\mathbf{T}_{3}$ is not heteroperturbative.
        Triangulate the unit sphere of $\protect\R^{3}$ with an equatorial
        strip of bisected darts as indicated on the left of the figure. By
        rotating the vertices and edges on the equator around the polar axis
        we obtain a combinatorially equivalent triangulation, as on the right.
        On the right, with respect to the geodesic metric on the sphere, some
        edges have increased in length when compared to the corresponding
        edges on the left, but no edge on the right has decreased in length
        when compared to the corresponding edge on the left.}
\end{figure}

An interesting heteroperturbative set of spherical triangulations
arises as a consequence of recent work by Winter on convex polytopes
from \cite{WinterPolytopes}. The following result is an easy consequence
of \cite[Corollary~4.13]{WinterPolytopes}:
\begin{thm}
    [Winter]\label{thm:Winter}Let $d\geq2$ and let $Y$ and $Z$ be
    combinatorially equivalent convex polytopes inscribed by the unit
    sphere of $\Rd$ so that the origin is contained in the interior of
    $Z$. If there exists an edge in $Z$ that is strictly shorter (with
    respect to Euclidean norm) than the corresponding edge of $Y$, then
    there exists an edge of $Z$ that is strictly longer (with respect
    to Euclidean norm) than the corresponding edge of $Y$.
\end{thm}

\begin{example}
    \label{exa:Winter-polytopes-heteroperturbative}Define the set $\mathbf{W}\subseteq\mathbf{T}$
    as all spherical triangulations $P\in\mathbf{T}$ of a unit sphere
    of any finite dimension $d\geq2$ for which the convex hull of the
    vertex set of $P$ forms a simplicial polytope whose face lattice
    is isomorphic, as a simplicial complex, to the spherical triangulation
    $P$ through central projection onto the unit sphere. By strict monotonicity
    of $\arcsin$ there exists a strictly monotone relationship between
    the length of an edge of the polytope and the geodesic length of the
    corresponding edge of the triangulation $P$. By Lemma~\ref{lem:spherical-triangulations-contain-center-in-interior},
    the convex hull of the vertex set of $P$ contains the center of the
    unit sphere as an interior point, and hence by Theorem~\ref{thm:Winter},
    $\mathbf{W}$ is seen to be heteroperturbative. Again, here the labels
    on vertices are superfluous.
\end{example}

We point out that $\mathbf{W}$ as defined in the previous example
is a proper subset of $\mathbf{T}$. Although it is easily seen that
$\mathbf{T}_{2}\subseteq\mathbf{W},$ the same is not true for $\mathbf{T}_{3}$,
(cf. Figure~\ref{fig:split-meridian}), as the face lattice of convex
hull of the vertices of a spherical triangulation in $\R^{3}$ can
be different to the vertex scheme of the triangulation.

\begin{figure}[b]
    \centering \includegraphics[width=0.7\textwidth]{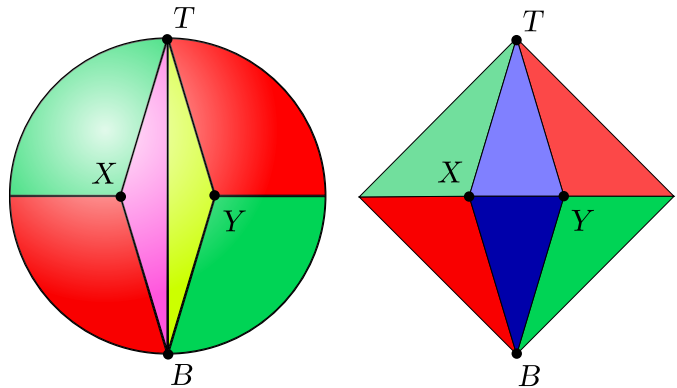}

    \caption{\label{fig:split-meridian}The vertex scheme of a spherical triangulation
        of the unit sphere in $\protect\R^{3}$ need not be isomorphic to
        face lattice of the polytope obtained as the convex hull of all vertices
        of the trangulation. Consider an octahedral triangulation of the unit
        sphere of $\protect\R^{3}$ with the vertex at the north pole moved
        slightly down on one of the meridians so as to not be lie on the vertical
        line through the origin. We split this meridian into two rather ``skinny''
        triangles as in the figure on the left. However, two vertices $X$
        and $Y$ can then ``see'' each other over the chord connecting the
        vertices $T$ and $B$. It is then seen that the face lattice of the
        polytope formed as the convex hull of the vertices of the triangulation
        is not isomorphic to the vertex scheme of the triangulation.}
\end{figure}

\medskip

A directly relevant set of spherical triangulations is the one we
now define.
\begin{defn}
    \label{def:define-Q}Define the set $\mathbf{Q}\subseteq\mathbf{T}$
    to be all labeled spherical triangulations $P\in\mathbf{T}$ for which
    there exists a monotone map $\rho:\N\cup\{0\}\to(0,\infty)$ so that,
    with $c\in\N\cup\{0\}$ denoting the label attached to the center
    of the unit sphere in $P$, both of the following conditions hold:
    \begin{enumerate}
        \item For all pairs of distinct vertices $v$ and $w$ of $P$ (whether
              connected by an edge or not), the geodesic distance in the unit sphere
              from $v$ to $w$ is at least $c_{\labelof w}^{\labelof v}|_{\rho}$.
        \item For all edges $vw$ in $P$, the geodesic length in the unit sphere
              of of $vw$ is exactly $c_{\labelof w}^{\labelof v}|_{\rho}$.
    \end{enumerate}
\end{defn}

Intuitively, such a labeled spherical triangulation triangulation
$P\in\mathbf{Q}$ describes an arrangement of spheres around a central
sphere, and with tangencies compatible with the labeled spherical
triangulation $P$ in the following way: The central sphere, centered
at the origin, has radius $\rho(c)$. By scaling, every vertex $v$
of $P$ corresponds to a point of tangency on the central sphere with
a neighboring sphere of radius $\rho(\labelof v)$. The first condition
of Definition~\ref{def:define-Q} ensures that the spheres of respective
radii $\rho(\labelof v)$ and $\rho(\labelof w)$ that are tangent
to the central sphere at points corresponding to respectively $v$
and $w$ have disjoint interiors. The second condition of Definition~\ref{def:define-Q}
ensures that the spheres with respective radii $\rho(\labelof v)$
and $\rho(\labelof w)$ that are tangent to the central sphere at
the points corresponding to $v$ and $w$ are tangent to each other.

By construction, canonically labeled spherical triangulations (cf.
Section~\ref{subsec:Canonical-labeling-and}) that arise in any compact
sphere packing $\packf p$ are members of $\mathbf{Q}$. A natural
question is whether or not $\mathbf{Q}$ is heteroperturbative.
\begin{conjecture}
    \label{conj:Q-is-heteroperturbative}The set $\mathbf{Q}$ of labeled
    spherical triangulations, as defined in Definition~\ref{def:define-Q},
    is heteroperturbative.
\end{conjecture}

A proof of Conjecture~\ref{conj:Q-is-heteroperturbative} together
with Theorem~\ref{thm:main-theorem} will immediately provide a proof
of Conjecture~\ref{conj:main-conjecture}.

\subsection{Canonical labelings, canonical realizers, canonical spherical triangulations,
    and packing codes determined by compact sphere packings\label{subsec:Canonical-labeling-and}}

Let $d,n\in\N$ with $d,n\geq2$. Let $\packf q$ be any collection
of spheres in $\Rd$ with $\radii(\packf q)=\{r_{0},r_{1},\ldots,r_{n-1}\}$
satisfying $0<r_{0}<r_{1}<\ldots<r_{n-1}$. By the \emph{canonical
    labeling of $\packf q$, }we mean attaching the label $j\in\{0,\ldots,n-1\}$
to the center of every sphere of radius $r_{j}$ in $\packf q$. With
$\Sigma:=\{0,\ldots,n-1\}$, we define the \emph{canonical realizer}
$\rho:\Sigma\to(0,\infty)$ of $\packf q$ as $\rho(j):=r_{j}$, for
all $j\in\Sigma$.

Let $\packf p$ be any compact sphere packing of $\Rd$ with $|\radii(\packf p)|=n$
that is canonically labeled by $\Sigma:=\{0,\ldots,n-1\}$. Fix any
sphere $A\in\packf p$. Since the underlying space of the packing
complex of $\packf p$ is homeomorphic to $\Rd$ through the identity
map, the center of the sphere $A$ is an interior point of the union
of the closed star of center of the sphere $A$ (as a set in $\Rd$
as a topological space). We obtain a labeled triangulation of the
sphere $A$ by centrally projecting the link of the center of $A$
in the packing complex of $\packf p$ onto the surface of the sphere
$A$ and carrying the labels of vertices along. Let $a\in\Sigma$
denote the label attached to the center of the sphere $A$. By scaling
and translating we obtain a labeled spherical triangulation of the
unit sphere of $\Rd$ in $\mathbf{T}$ and we call this the \emph{canonical
    labeled triangulation associated to} \emph{the sphere }$A\in\packf p$.
Let $T_{A}$ be the abstract simplicial complex obtained as the vertex
scheme of the canonical labeled triangulation of the sphere $A$ by
forgetting all geometric information, but retaining all labels on
the vertices. Then $a:T_{A}$ is a $d$-dimensional packing code over
$\Sigma$, and we call \emph{$a:T_{A}$ }the \emph{canonical packing
    code associated to the sphere $A$. }We define $\code(A):=a:T_{A}$
and $\pcodes(\packf p):=\set{\code(A)\in\codes(\Sigma)}{A\in\packf p}.$

\section{Fundamental sets of packing codes\label{sec:Fundamental-sets-of}}

In this section, we define what we call fundamental sets of packing
codes. These structures are always present as a subset of all the
packing codes obtained from a compact sphere packing (cf. Theorem~\ref{thm:packings-determine-fundamental-set}).
In Figure~\ref{fig:pack-five-labelled} we display an example of
a fundamental set determined by a compact packing in two dimensions.

Fundamental sets allow us to firstly show that the combinatorics of
a compact packing associated to a heteroperturbative set uniquely
determine the radii of the spheres occurring in the packing (cf. Section~\ref{sec:Uniqueness}),
and to secondly establish The Bootstrapping Lemma in Section~\ref{sec:The-bootstrapping-lemma}
relating ratios of values of realizers under the condition of existence
of appropriate labeled spherical triangulations from a heteroperturbative
set who are combinatorially equivalent with identical labels to the
codes in the fundamental set.
\begin{defn}
    Let $d,n\in\N$ with $d,n\geq2$ and $\Sigma:=\{0,\ldots,n-1\}$.
    Let $C\subseteq\codes(\Sigma)$. We will say that $C$ is\emph{ a
        fundamental set (of packing codes) }if $\set{c\,}{c:T\in C}=\{0,\ldots,n-2\}$
    and for every non-empty set $K\subseteq\{0,\ldots,n-2\}$ there exists
    a code $c:T\in C$ so that $c\in K$ and there exists a vertex from
    $T$ that has label that is not an element of $K$, i.e., $\codeneighborlabels(T)\backslash K\neq\emptyset$.
\end{defn}

{}
\begin{thm}
    \label{thm:packings-determine-fundamental-set}Let $d,n\in\N$ with
    $d,n\geq2$ and $\Sigma:=\{1,\ldots,n-1\}$. Let $\packf p$ be a
    canonically labeled compact sphere packing in $\Rd$ with $|\radii(\packf p)|=n$.
    The set $C:=\set{c:T\in\pcodes(\packf p)}{c\leq n-2}$ is a fundamental
    set.
\end{thm}

\begin{proof}
    Since $|\radii(\packf p)|=n$, we have $\set c{c:T\in C}=\{0,\ldots,n-2\}$.
    Suppose there exists a non-empty set $K\subseteq\{0,\ldots,n-2\}$
    so that for all $c:T\in C$ with $c\in K$ we have that $\codeneighborlabels(T)\backslash K=\emptyset$.
    As there exists a sphere labeled by an element from $K$, and every
    sphere in $\packf p$ labeled by an element from $K$ only has neighbors
    labeled by elements from $K$, we conclude that all spheres in $\packf p$
    must be labeled by elements from $K$. But this yields the contradiction
    $n=|\radii(\packf p)|=|K|<n$. We conclude that $C$ is fundamental.
\end{proof}
\begin{figure}
    \centering \includegraphics[width=100mm]{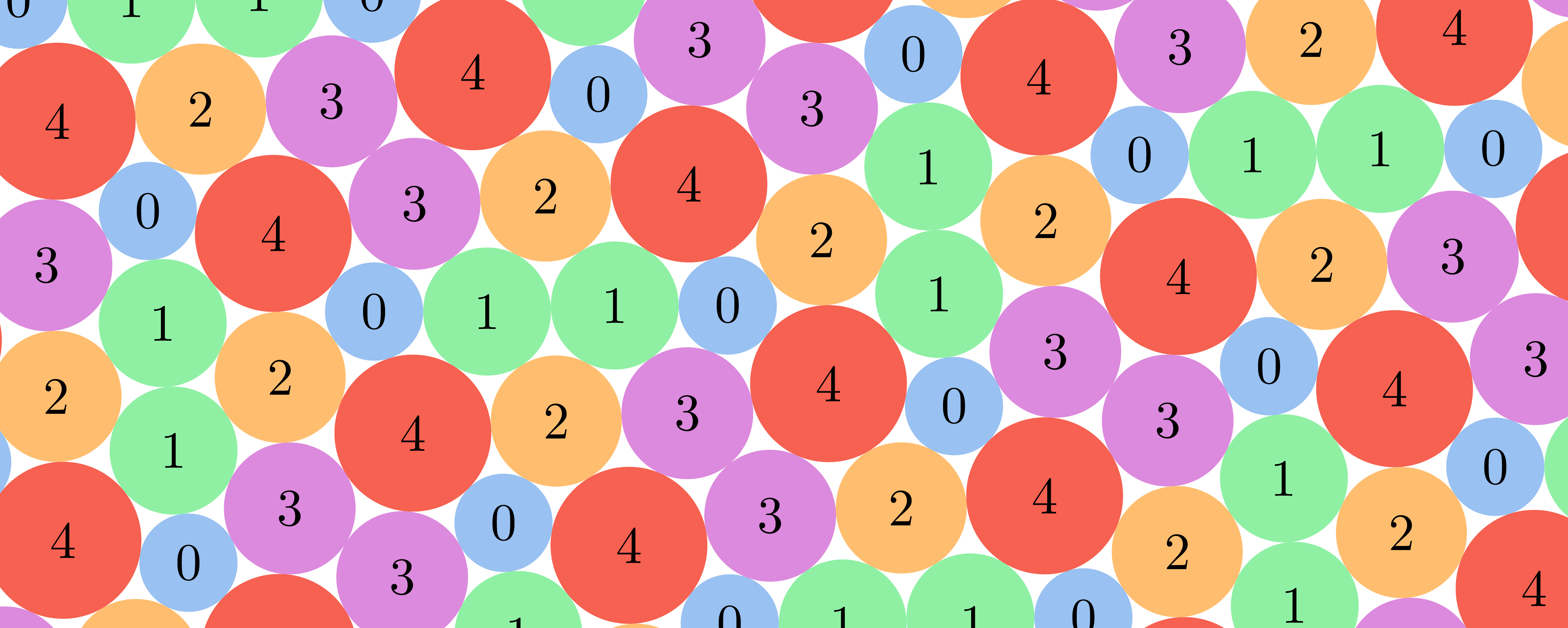}\caption{\label{fig:pack-five-labelled}The figure displays a compact packing
        $\protect\packf p$ of $\protect\R^{2}$ from \cite{fernique2023-five-packing}
        with five sizes of discs. The canonical labeling of the packing is
        overlaid. Packing codes associated to compact disc packings in two
        dimensions have cycles as their neighbor complexes. We can therefore
        succinctly express the neighbor complexes of such codes as just the
        string of labels on the vertices in the order that the vertices occur
        in the cycle. From the figure, we thus read off the set $\protect\pcodes(\protect\packf p)$
        (modulo reflection and rotation) as:
        \[
            \left\{ \protect\begin{array}{c}
                0:43142,\ 1:421230,\protect  \\
                2:431140,\ 3:434210,\protect \\
                4:3320120
                \protect\end{array}\right\} .
        \]
        The conclusion of Theorem~\ref{thm:packings-determine-fundamental-set}
        is that $\protect\set{c:T\in\protect\pcodes(\protect\packf p)}{c\protect\leq3}$
        is a fundamental set. This can be easily verified directly.}
\end{figure}

\section{Uniqueness of radii in compact packings associated to a heteroperturbative
  set\label{sec:Uniqueness}}

In this section our goal is to prove Theorem~\ref{thm:canonical-realizers-are-unique-for-packings}
that shows, for some fixed heteroperturbative set $\mathbf{S}$, that
the combinatorics of a compact packing associated to $\mathbf{S}$
uniquely determines the radii of the spheres in the packing.

The following lemma should be seen as a generalization of the result
\cite[Theorem~4.2]{Messerschmidt2d} to higher dimensions. The argument
is essentially the same, proceeding through application of Lemma~\ref{lem:realize_angle_symbol_strict_monotone}
and leveraging the defining property of heteroperturbative sets.
\begin{lem}
    \label{lem:uniqueness-lemma}Let $d,n\in\N$ with $d,n\geq2$ and
    $\Sigma:=\{0,\ldots,n-1\}$. Let $\mathbf{S}$ be a heteroperturbative
    set of labeled spherical triangulations of unit spheres. Let $C\subseteq\codes(\Sigma)$
    be a fundamental set. Let $\rho,\sigma\in\{\tau:\Sigma\to(0,\infty):\tau(n-1)=1\}$
    and assume, for all $c:T_{c}\in C$, that there exist spherical triangulations
    $P_{c},Q_{c}\in\mathbf{S}$ that are both combinatorially equivalent
    to $T_{c}$ with identical labels and are such that, for every edge
    $vw$ in $T_{c}$, the corresponding edges of the triangulations $P_{c}$
    and $Q_{c}$ respectively have geodesic lengths $c_{\labelof w}^{\labelof v}|_{\rho}$
    and $c_{\labelof w}^{\labelof v}|_{\sigma}$. Then $\rho=\sigma$.
\end{lem}

\begin{proof}
    We proceed by contradiction: Suppose that $\rho\neq\sigma$. Define
    $t_{0}:=\sup\{t>0:\forall j\in\Sigma,\ t\sigma(j)<\rho(j)\}$ and
    $J:=\{j\in\Sigma:t_{0}\sigma(j)=\rho(j)\}$. Note that the set $J$
    is not empty, else $t_{0}$ cannot be the supremum of the set $\{t>0:\forall j\in\Sigma,\ t\sigma(j)<\rho(j)\}.$
    Further, since $\rho(n-1)=\sigma(n-1)=1$, we have that $n-1\notin J$.
    Since $C$ is fundamental, there exists a code $k:T_{k}\in C$ with
    $k\in J$ and some $p\in\vertexset(T_{k})$ so that $\labelof p\notin J$.
    Define $\nu:=\rho-t_{0}\sigma$. The map $\nu$ takes on non-negative
    values and the support of $\nu$ is exactly $\Sigma\backslash J$.
    By Lemma~\ref{lem:realize_angle_symbol_strict_monotone}, for every
    edge $vw$ of $T_{k}$,
    \[
        k_{\labelof w}^{\labelof v}|_{\sigma}=k_{\labelof w}^{\labelof v}|_{t_{0}\sigma}=k_{\labelof w}^{\labelof v}|_{t_{0}\sigma+0\nu}\leq k_{\labelof w}^{\labelof v}|_{t_{0}\sigma+1\nu}=k_{\labelof w}^{\labelof v}|_{\rho}.
    \]
    Therefore no edge of $P_{k}$ is strictly shorter that the corresponding
    edge of $Q_{k}$. Furthermore, since $\labelof p\notin J$, again
    by Lemma~\ref{lem:realize_angle_symbol_strict_monotone}, at least
    one of the above inequalities is strict, implying that an edge of
    $P_{k}$ is strictly longer than the corresponding edge of $Q_{k}$.
    Hence $P_{k}$ and $Q_{k}$ are not edge-isometric, while living in
    the same equivalence class in the heteroperturbative set $\mathbf{S}$.
    This implies the existence of an edge in $P_{k}$ that is strictly
    shorter that the corresponding edge of $Q_{k}$, contrary to our earlier
    remark that this is not the case. Consequently the supposition $\rho\neq\sigma$
    is false, and we conclude that $\rho=\sigma$.
\end{proof}
\begin{lem}
    \label{lem:packing-determines-triangulations-from-Q}Let $d,n\in\N$
    with $d,n\geq2$ and define $\Sigma:=\{0,\ldots,n-1\}$. Let $\packf p$
    be a canonically labeled compact sphere packing in $\Rd$ with $|\radii(\packf p)|=n$
    and normalized so that $\max\radii(\packf{p\packf )}=1$. Let $\rho:\Sigma\to(0,1]$
    denote the canonical realizer of the packing $\packf p$. The canonical
    labeled triangulation associated to any sphere in $\packf p$ (by
    using the canonical realizer $\rho$) is an element of the set $\mathbf{Q}$,
    as in Definition~\ref{def:define-Q}.
\end{lem}

\begin{proof}
    This follows from construction as discussed in Section~\ref{subsec:spherical-triangulations-and-heteroperturbative-sets}.
\end{proof}
\begin{thm}
    \label{thm:canonical-realizers-are-unique-for-packings}Let $d,n\in\N$
    with $d,n\geq2$ and define $\Sigma:=\{0,\ldots,n-1\}$. Let $\mathbf{S}$
    be a heteroperturbative set of labeled spherical triangulations of
    unit spheres. Let $\packf p$ be a canonically labeled compact sphere
    packing in $\Rd$ with $|\radii(\packf p)|=n$, normalized so that
    $\max\radii(\packf{p\packf )}=1$, and so that the canonical labeled
    spherical triangulation associated to every sphere from $\packf p$
    is an element of $\mathbf{S}$. Let $C\subseteq\pcodes(\packf p)$
    be a fundamental set. The canonical realizer $\rho:\Sigma\to(0,1]$
    of the packing $\packf p$ is the unique map in $\{\tau:\Sigma\to(0,\infty):\tau(n-1)=1\}$
    so that, for every $c:T_{c}\in C$, there exists a spherical triangulation
    $P_{c}\in\mathbf{S}$ that is combinatorially equivalent to $T_{c}$
    with identical labels and is such that for every edge $vw$ of $T_{c}$
    the corresponding edge in $P_{c}$ has length $c_{\labelof w}^{\labelof v}|_{\rho}$.
\end{thm}

\begin{proof}
    By Lemma~\ref{lem:packing-determines-triangulations-from-Q}, the
    canonical realizer $\rho$ of $\packf p$ is such that for every $c:T_{c}\in C$,
    there exists a spherical triangulation $P_{c}\in\mathbf{G}$ (the
    canonical labeled spherical triangulation of some sphere from $\packf p$)
    that is combinatorially equivalent to $T_{c}$ with identical labels
    and is such that, for every edge $vw$ of $T_{c},$ the corresponding
    edge in $P_{c}$ has length $c_{\labelof w}^{\labelof v}|_{\rho}$.
    But the canonical labeled spherical triangulation of all spheres in
    $\packf p$ are all assumed to be elements of $\mathbf{S}$. Therefore
    the canonical realizer $\rho$ satisfies the stated condition.

    By Lemma~\ref{lem:uniqueness-lemma} the canonical realizer $\rho$
    is the only map in $\{\tau:\Sigma\to(0,\infty):\tau(n-1)=1\}$ satisfying
    this condition.
\end{proof}

\section{The bootstrapping lemma\label{sec:The-bootstrapping-lemma}}

This section sees the proof of The Bootstrapping Lemma (Lemma~\ref{lem:bootstrapping-lemma2-1}).
This result is a crucial ingredient in the induction step of the strong
induction performed in proving the main technical result of this paper
(Lemma~\ref{lem:essential-sets-finite-induction-step}). Lemma~\ref{lem:bootstrapping-lemma2-1}
is analogous to \cite[Lemma~5.1]{Messerschmidt2d}.

We introduce the following admittedly tortuously abused notation for
which the authors apologize. The reason for its introduction is to
be able to succinctly express how realizers relate to labeled spherical
triangulations of the unit sphere in this section and in the subsequent
section.
\begin{defn}
    \label{def:abused-notation}Let $\Sigma$ be any set of symbols and
    let $c:T_{c}\in\codes(\Sigma)$. Let $P\in\mathbf{T}$ be any labeled
    spherical triangulation of the unit sphere and let $\rho:\Sigma\to(0,\infty)$
    be any map.

    By the notation
    \[
        \rho\underleq{c:T_{c}}P
    \]
    we mean the following: Firstly, $P$ is combinatorially equivalent
    to $T_{c}$ with identical labelings, and secondly, for all \textbf{pairs
        of distinct vertices} $v,w$ of $T_{c}$ (not necessarily connected
    by an edge), the geodesic distance in the unit sphere between the
    vertices corresponding to $v$ and $w$ in $P$ is at least $c_{\labelof w}^{\labelof v}|_{\rho}$.

    By the notation
    \[
        \rho\underleqtop{c:T_{c}}P
    \]
    we mean the following: Firstly, $P$ is combinatorially equivalent
    to $T_{c}$ with identical labelings, and secondly, for all \textbf{edges}
    $vw$ of $T_{c}$, the geodesic length in the unit sphere of the corresponding
    edge in $P$ is least $c_{\labelof w}^{\labelof v}|_{\rho}$.

    The meaning of the notation
    \[
        P\underleq{c:T_{c}}\rho\qquad\text{and}\qquad P\underleqtop{c:T_{c}}\rho
    \]
    is defined to be the same as above, but with the words `at least'
    are replaced with the words `at most'.
\end{defn}

\begin{lem}[The Bootstrapping Lemma]
    \label{lem:bootstrapping-lemma2-1}Let $d,n\in\N$ with $d,n\geq2$
    and $\Sigma:=\{0,\ldots,n-1\}$. Let $\mathbf{S}$ be a heteroperturbative
    set of labeled spherical triangulations of unit spheres. Let $C\subseteq\codes(\Sigma)$
    be a fundamental set. If $\rho,\sigma:\Sigma\to(0,\infty)$ are such
    that, for all $c:T_{c}\in C$ there exist triangulations $P_{c},Q_{c}\in\mathbf{S}$
    so that $\rho\underleqtop{c:T_{c}}P_{c}$ and $Q_{c}\underleqtop{c:T_{c}}\sigma$,
    as in Definition~\ref{def:abused-notation}, then
    \[
        \frac{\sigma(n-2)}{\sigma(n-1)}\leq\frac{\rho(n-2)}{\rho(n-1)}.
    \]
\end{lem}

\begin{proof}
    Assume $\rho,\sigma:\Sigma\to(0,\infty)$ are such that, for all $c:T_{c}\in C$
    there exist triangulations $P_{c},Q_{c}\in\mathbf{S}$ so that $\rho\underleqtop{c:T_{c}}P_{c}$
    and $Q_{c}\underleqtop{c:T_{c}}\sigma$.

    If there exists some $t\in(0,\infty)$ so that $\rho=t\sigma$, then
    we immediately obtain
    \[
        \frac{\sigma(n-2)}{\sigma(n-1)}=\frac{t\sigma(n-2)}{t\sigma(n-1)}=\frac{\rho(n-2)}{\rho(n-1)}.
    \]
    We hence assume, for all $t\in(0,\infty)$ that $\rho\neq t\sigma$.

    Define $t_{0}:=\sup\{t\in(0,\infty):\forall j\in\Sigma,\ t\sigma(j)<\rho(j)\}>0$
    and
    \[
        J:=\{j\in\Sigma:t_{0}\sigma(j)=\rho(j)\}.
    \]
    We note that $J$ is not empty, otherwise $t_{0}$ cannot be the
    supremum of the set $\{t\in(0,\infty):\forall j\in\Sigma,\ t\sigma(j)<\rho(j)\}$.
    Also, since for all $t\in(0,\infty)$ we have $\rho\neq t\sigma$,
    we have that $J$ is a proper subset of $\Sigma$.

    There are two cases: Either $n-1\in J$ or $n-1\notin J$. We show
    that we must have $n-1\in J$ through obtaining a contradiction in
    the other case.

    Suppose that $n-1\notin J$, hence $J\subseteq\{0,\ldots,n-2\}$.
    Since $C$ is a fundamental set, there exists some $k:T_{k}\in C$
    that has some vertex $p$ in $T_{k}$ with $\labelof p\notin J$.
    Define $\nu:=\rho-t_{0}\sigma$, which can be seen to be non-negative
    and have support exactly equal to $\Sigma\backslash J$. By Lemma~\ref{lem:realize_angle_symbol_strict_monotone},
    for all edges $vw$ in $T_{k}$, we have
    \[
        k_{\lambda w}^{\lambda v}|_{\sigma}=k_{\lambda w}^{\lambda v}|_{t_{0}\sigma}=k_{\lambda w}^{\lambda v}|_{t_{0}\sigma+0\nu}\leq k_{\lambda w}^{\lambda v}|_{t_{0}\sigma+1\nu}=k_{\lambda w}^{\lambda v}|_{\rho}.
    \]
    Since $\rho\underleqtop{c:T_{c}}P_{k}$ and $Q_{k}\underleqtop{c:T_{c}}\sigma$,
    by the above chain of inequalities, all edges of $P_{k}$ are longer
    or equal in length to the corresponding edges in $Q_{k}$. Since $\labelof p\notin J$,
    again by Lemma~\ref{lem:realize_angle_symbol_strict_monotone}, at
    least one of the above inequalities is strict. Therefore $P_{k}$
    and $Q_{k}$ are not edge-isometric, but belong to the same equivalence
    class in the heteroperturbative set $\mathbf{S}$, and therefore there
    exists an edge of $P_{k}$ that is strictly shorter than the corresponding
    edge of $Q_{k}$. This is in contradiction with the earlier remark
    that all edges of $P_{k}$ are longer or equal in length to the corresponding
    edges in $Q_{k}$. We conclude that we must have $n-1\in J$.

    With $n-1\in J$, we have $t_{0}\sigma(n-1)=\rho(n-1)$ and $t_{0}\sigma(n-2)\leq\rho(n-2)$.
    Hence, we obtain
    \[
        \frac{\sigma(n-2)}{\sigma(n-1)}=\frac{t_{0}\sigma(n-2)}{t_{0}\sigma(n-1)}=\frac{t_{0}\sigma(n-2)}{\rho(n-1)}\leq\frac{\rho(n-2)}{\rho(n-1)}.\qedhere
    \]
\end{proof}

\section{Essential sets and proof of the main result\label{sec:Essential-sets}}

In this section, we prove our main result Theorem~\ref{thm:main-theorem}
through the main technical results of this paper. The arguments in
this section have the proof of Lemma~\ref{lem:FerniqueBabyBootstrap}
in \cite[Lemma~6.1]{Fernique2021} as their germ, and closely follows
\cite[Section~6]{Messerschmidt2d}.

We start with the following definition:
\begin{defn}
    \label{def:essential-set}Let $d,n\in\N$ with $d,n\geq2$ and $\Sigma:=\{0,\ldots,n-1\}$.
    Let $\mathbf{S}$ be a heteroperturbative set of labeled spherical
    triangulations of unit spheres. We say a fundamental set $C\subseteq\codes(\Sigma)$
    is an \emph{$n$-essential set for $\mathbf{S}$ in dimension $d$
    }if there exist monotone maps $\rho,\sigma:\Sigma\to(0,\infty)$ so
    that, for every $c:T_{c}\in C$, the set
    \[
        \set{P\in\mathbf{S}}{\rho\underleq{c:T_{c}}P\underleqtop{c:T_{c}}\sigma}
    \]
    is non-empty. We denote the set of all $n$-essential sets for $\mathbf{S}$
    in dimension $d$ by $\essential_{d,n}(\mathbf{S}).$
\end{defn}

Definition~\ref{def:essential-set} places further conditions on
fundamental sets. In showing that these extra conditions are automatically
satisfied by fundamental sets that arise from compact packings associated
to the heteroperturbative set $\mathbf{S}$, we can show that the
cardinality of the set $\Pi_{d,n}(\mathbf{S})$ is bounded above by
the cardinality of the set $\essential_{d,n}(\mathbf{S})$ (cf. Lemma~\ref{lem:Pidn-cardinality-bounded-by-n-essential-sets}).
The aim in this section is thus to show, for $d,n\in\N$ with $d,n\geq2$
and any heteroperturbative set $\mathbf{S}$, that the set $\essential_{d,n}(\mathbf{S})$
is finite (cf. Corollaries~\ref{cor:ess-finite} and~\ref{cor:main-result}).
This is achieved through a strong induction arglument by showing first
showing $\essential_{d,2}(\mathbf{S})$ is finite, and subsequently
showing that $\essential_{d,n}(\mathbf{S})$ is finite under assumption
that all the `lower' sets $\essential_{d,k}(\mathbf{S})$ with $k\in\{2,3,\ldots,n-1\}$
are finite (cf. Lemmas~\ref{lem:2-ess-finite} and~\ref{lem:essential-sets-finite-induction-step}).

\medskip

Before we continue, at the request of the anonymous referee, we explain
some more of the intuition behind this section. Admittedly, our choice
of Definition~\ref{def:essential-set} may appear somewhat opaque.
The reason for this apparent opacity is likely the process at arriving
at this definition through the trial and error approach taken toward
proving the results in this section \emph{without }first having a
concrete definition of an essential set. Only after the results in
this section had been refined by trial and error, could one see exactly
what the ``correct''\footnote{By ``correct'' in this context we mean: Firstly, the cardinality of
    $\essential_{d,n}(\mathbf{S})$ bounds the cardinality of $\Pi_{d,n}(\mathbf{S})$.
    And secondly, the definition of $\essential_{d,n}(\mathbf{S})$ allows
    for proof of the main result through application of The Bootstrapping
    Lemma in as elegant a manner as is possible.} definition of an essential set \emph{must }be. The conditions placed
on the elements of $\essential_{d,n}(\mathbf{S})$ by Definition~\ref{def:essential-set}
along with the notations, $\blacktriangleleft$ and $\vartriangleleft$,
were then chosen in exactly such a way so as to be able to precisely
relate an element in $\essential_{d,n}(\mathbf{S})$ to elements in
the `lower' sets $\essential_{d,k}(\mathbf{S})$ for $k\in\{2,3,\ldots,n-1\}$
through the $\downarrow$-operation (cf. Definition~\ref{def:down-arrow}
and Lemmas~\ref{lem:fundamentally-downward-preserved} and~\ref{lem:down-essential}).
Furthermore, the conditions from Definition~\ref{def:essential-set}
were chosen such that, in relating an element $C$ from $\essential_{d,n}(\mathbf{S})$
through the $\downarrow$-operation to elements in the `lower' sets
$\essential_{d,k}(\mathbf{S})$ for $k\in\{2,3,\ldots,n-1\}$, allows
for obtaining information on the map $\rho$ for $C$ in terms of
the $\sigma$'s from the lower essential sets through application
of The Bootstrapping Lemma (Lemma~\ref{lem:bootstrapping-lemma2-1}).
This idea forms the heart of the proof of Lemma~\ref{lem:essential-sets-finite-induction-step}.

\medskip

We now proceed with the technical details of this section.
\begin{lem}
    \label{lem:Pidn-cardinality-bounded-by-n-essential-sets}Let $d,n\in\N$
    with $d,n\geq2$. Let $\mathbf{S}$ be a heteroperturbative set of
    labeled spherical triangulations of unit spheres. The set $\Pi_{d,n}(\mathbf{S})$
    has cardinality at most that of the set $\essential_{d,n}(\mathbf{S}).$
\end{lem}

\begin{proof}
    Let $\packf p$ be a canonically labeled compact sphere packing in
    $\Rd$ with $|\radii(\packf p)|=n$ so that $\max\radii(\packf p)=1$
    and such that the canonical labeled triangulation of the unit sphere
    associated to every sphere in $\packf p$ is an element of $\cal{\mathbf{S}}$.
    By Theorem~\ref{thm:packings-determine-fundamental-set} there exists
    a fundamental set $C\subseteq\pcodes(\packf p)$. Taking both maps
    $\rho$ and $\sigma$ in Definition~\ref{def:essential-set} as equal
    to the canonical realizer of $\packf p$ and applying Lemma~\ref{lem:packing-determines-triangulations-from-Q},
    we see that for all $c:T_{c}\in C$, the set $\set{P\in\cal{\mathbf{S}}}{\rho\underleq{c:T_{c}}P\underleqtop{c:T_{c}}\sigma}$
    is non-empty and hence $C$ is an $n$-essential set for $\mathbf{S}$
    in dimension $d$.

    Let $\packf q$ be another canonically labeled compact sphere packing
    in $\Rd$ satisfying $|\radii(\packf q)|=n$ with $\max\radii(\packf q)=1$,
    is such that $C\subseteq\pcodes(\packf q)$, and has the canonical
    labeled triangulation of the unit sphere associated to every sphere
    in $\packf q$ being an element of $\cal{\mathbf{S}}$. By Lemma~\ref{thm:canonical-realizers-are-unique-for-packings},
    the canonical realizers of the packings $\packf p$ and $\packf q$
    are equal.

    We conclude that each compact packing associated to $\mathbf{S}$
    determines at least one element of the set $\essential_{d,n}(\mathbf{S})$,
    and each element of the set $\essential_{d,n}(\mathbf{S})$ determines
    at most one element of $\Pi_{d,n}(\mathbf{S}).$ Therefore $\Pi_{d,n}(\mathbf{S})$
    has cardinality at most that of the cardinality of the set $\essential_{d,n}(\mathbf{S})$.
\end{proof}
\begin{lem}
    \label{lem:2-ess-finite}Let $d\in\N$ with $d\geq2$. Let $\mathbf{S}$
    be a heteroperturbative set of labeled spherical triangulations of
    unit spheres. The set $\essential_{d,2}(\mathbf{S})$ of all $2$-essential
    sets for $\mathbf{S}$ in dimension $d$ is finite.
\end{lem}

\begin{proof}
    Let $C\subseteq\codes(\{0,1\})$ be any $2$-essential set for $\mathbf{S}$
    in dimension $d$. Since $C$ is fundamental, for all $c:T_{c}\in C$
    we have $c=0$. Further, there exist monotone maps $\rho,\sigma:\{0,1\}\to(0,\infty)$
    so that for all $c:T_{c}\in C$ the set $\set{P\in\mathbf{S}}{\rho\underleq{c:T_{c}}P\underleqtop{c:T_{c}}\sigma}$
    is non-empty. Since $\rho$ is monotone, we have that $\pi/3\leq0_{0}^{0}|_{\rho},\ 0_{0}^{1}|_{\rho},\ 0_{1}^{1}|_{\rho}$.
    Hence, for all $c:T_{c}\in C$, every spherical triangulation from
    the set $\set{P\in\mathbf{S}}{\rho\underleq{c:T_{c}}P\underleqtop{c:T_{c}}\sigma}$
    has a geodesic distance of at least $\pi/3$ between all pairs of
    distinct vertices.

    By compactness of the unit sphere $S\subseteq\Rd$, we let $N\in\N$
    be the least cardinality of a cover of $S$ by geodesic open balls
    of radius $\pi/6$ with centers on $S$. But then for all codes $c:T_{c}\in C$
    no labeled spherical triangulation from the set $\set{P\in\mathbf{S}}{\rho\underleq{c:T_{c}}P\underleqtop{c:T_{c}}\sigma}$
    can have more than $N$ vertices otherwise there would exist two distinct
    vertices strictly closer than $\pi/3$. Hence in every code $0:T_{c}\in C$,
    the neighbor complex $T$ has at most $N$ vertices. We conclude that
    the set of all $2$-essential sets or $\mathbf{S}$ in dimension $d$,
    $\essential_{d,2}(\mathbf{S})$, has cardinality at most that of the
    powerset of the finite set $\set{0:T\in\codes(\{0,1\})}{|\text{vertex set of }T|\leq N}.$
\end{proof}
\begin{defn}
    \label{def:down-arrow}Let $\Sigma$ denote any totally ordered set
    of symbols. For any structure $T$ labeled by elements from $\Sigma$,
    for $s\in\Sigma$, by $T\downarrow_{s}$ we mean a relabeling of $T$
    in which all labels occurring in $T$ that are strictly larger than
    $s$ are replaced by $s$.
\end{defn}

\begin{lem}
    \label{lem:fundamentally-downward-preserved}For any $k\in\N$, define
    $\Sigma_{k}:=\{0,\ldots,k-1\}$. Let $d,n\in\N$ with $d,n\geq2$.
    If $C\subseteq\codes(\Sigma_{n})$ is fundamental, then for any $k\in\N$
    with $2\leq k\leq n$, the set $C_{k}:=\set{c:T_{c}\downarrow_{k-1}}{c:T_{c}\in C,\ c\leq k-2}\subseteq\codes(\Sigma_{k})$
    is fundamental.
\end{lem}

\begin{proof}
    Let $2\leq k\leq n.$ Since $C$ is fundamental we have that, $\set{c\,}{c:T\in C_{k}}=\{0,\ldots,k-2\}$.
    Let $K$ be a non-empty subset of $\{0,\ldots,k-2\}\subseteq\{0,\ldots,n-2\}$.
    Since $C$ is fundamental there exists some $c:T\in C$ so that $c\in K$,
    but with some $p\in\codeneighborlabels(T)\backslash K$. There are
    two cases, $p\geq k-1$ and $p<k-1$. In the case that $p\geq k-1$,
    then $p\downarrow_{k-1}=k-1\notin K$, so that $c:T\downarrow_{k-1}\in C_{k}$
    is such that $c\in K$, with $p\downarrow_{k-1}\in\codeneighborlabels(T\downarrow_{k-1})\backslash K$.
    On the other hand, if $p<k-1$, then $p\downarrow_{k-1}=p\notin K$,
    and $c:T\downarrow_{k-1}\in C_{k}$ is such that $p=p\downarrow_{k-1}\in\codeneighborlabels(T\downarrow_{k-1})\backslash K$.
    We conclude that $C_{k}$ is fundamental.
\end{proof}
\begin{lem}
    \label{lem:down-essential} For any $k\in\N$, define $\Sigma_{k}:=\{0,\ldots,k-1\}$.
    Let $d,n\in\N$ with $d,n\geq2$. Let $\mathbf{S}$ be a heteroperturbative
    set of labeled spherical triangulations of unit spheres. If $C\subseteq\codes(\Sigma_{n})$
    is an element of $\essential_{d,n}(\mathbf{S})$, then for any $k\in\N$
    with $2\leq k\leq n$, the set $C_{k}:=\set{c:T_{c}\downarrow_{k-1}}{c:T_{c}\in C,\ c\leq k-2}\subseteq\codes(\Sigma_{k})$
    is an element of $\essential_{d,k}(\mathbf{S})$.
\end{lem}

\begin{proof}
    Assume that $C\subseteq\codes(\Sigma_{n})$ is an $n$-essential set
    for $\mathbf{S}$ in dimension $d$. Let $\sigma,\rho:\Sigma_{n}\to(0,\infty)$
    be monotone maps so that for all $c:T_{c}\in C$ the set
    \[
        \set{P\in\mathbf{S}}{\rho\underleq{c:T_{c}}P\underleqtop{c:T_{c}}\sigma}
    \]
    is non-empty.

    Let $k\in\{2,\ldots,n\}$, and define $C_{k}:=\set{c:T_{c}\downarrow_{k-1}}{c:T_{c}\in C,\ c\leq k-2}\subseteq\codes(\Sigma_{k})$.
    By Lemma~\ref{lem:fundamentally-downward-preserved}, the set $C_{k}$
    is fundamental.\textbf{ }Define $\rho',\sigma':\Sigma_{k}\to(0,\infty)$
    as

    \[
        \rho'(s):=\begin{cases}
            \rho(s)   & s<k-1 \\
            \rho(k-1) & s=k-1
        \end{cases}\qquad(s\in\Sigma_{k})
    \]
    and
    \[
        \sigma'(s):=\begin{cases}
            \sigma(s)   & s<k-1 \\
            \sigma(n-1) & s=k-1
        \end{cases}\qquad(s\in\Sigma_{k}).
    \]
    Since $\rho$ and $\sigma$ are both monotone, so are $\rho'$ and
    $\sigma'$. Furthermore, we note, for all $c\in\{0,\ldots,k-2\}$
    and $a,b\in\Sigma$, that
    \[
        c_{b}^{a}\downarrow_{k-1}\,|_{\rho'}\leq c_{b}^{a}|_{\rho}\qquad\text{and}\qquad c_{b}^{a}|_{\sigma}\leq c_{b}^{a}\downarrow_{k-1}\,|_{\sigma'}\ .
    \]
    Therefore, for all $c:T_{c}\downarrow_{k-1}\in C_{k}$,
    \[
        \emptyset\neq\set{P\in\mathbf{S}}{\rho\underleq{c:T_{c}}P\underleqtop{c:T_{c}}\sigma}\subseteq\set{P\in\mathbf{S}}{\rho'\underleq{c:T_{c}\downarrow_{k-1}}P\underleqtop{c:T_{c}\downarrow_{k-1}}\sigma'}.
    \]
    Therefore $C_{k}$ is a $k$-essential set for $\mathbf{S}$ in dimension
    $d$.
\end{proof}
\begin{lem}
    \label{lem:essential-sets-finite-induction-step}Let $d,n\in\N$ with
    $d,n\geq2$. Let $\mathbf{S}$ be a heteroperturbative set of labeled
    spherical triangulations of unit spheres. If, for all $k\in\{2,\ldots,n-1\}$,
    the set $\essential_{d,k}(\mathbf{S})$ is finite, then the set $\essential_{d,n}(\mathbf{S})$
    is finite.
\end{lem}

\begin{proof}
    For any $k\in\N$, define $\Sigma_{k}:=\{0,\ldots,k-1\}$.

    Assume for all $k\in\{2,\ldots,n-1\}$ that the set $\essential_{d,k}(\mathbf{S})$
    is finite. For every $k\in\{2,\ldots,n-1\}$ and every one of the
    finitely many $D\in\essential_{d,k}(\mathbf{S})$, we fix monotone
    maps $\sigma_{D},\rho_{D}:\Sigma_{k}\to(0,\infty)$ as in Definition~\ref{def:essential-set}
    so that, for all $c:T_{c}\in D$, we have
    \[
        \set{P\in\mathbf{S}}{\rho_{D}\underleq{c:T_{c}}P\underleqtop{c:T_{c}}\sigma_{D}}\neq\emptyset.
    \]
    For every $k\in\{2,\ldots,n-1\}$, define
    \[
        K_{k-2}:=\min\set{\frac{\sigma_{D}(k-2)}{\sigma_{D}(k-1)}}{D\in\essential_{d,k}(\mathbf{S})}\in(0,1].
    \]

    Let $C\in\codes(\Sigma_{n})$ be any element of the set $\essential_{d,n}(\mathbf{S})$.
    By Definition~\ref{def:essential-set}, there exist monotone maps
    $\sigma,\rho:\Sigma_{n}\to(0,\infty)$ so that, for every $c:T_{c}\in C$,
    we have
    \[
        \set{P\in\mathbf{S}}{\rho\underleq{c:T_{c}}P\underleqtop{c:T_{c}}\sigma}\neq\emptyset.
    \]

    For each $k\in\{2,\ldots,n-1\}$, by Lemma~\ref{lem:down-essential},
    the set
    \[
        C_{k}:=\set{c:T_{c}\downarrow_{k-1}}{c:T_{c}\in C,\ c\leq k-2}\subseteq\codes(\Sigma_{k})
    \]
    is an element of $\essential_{d,k}(\mathbf{S})$, and hence, for
    all $c:T_{c}\downarrow_{k-1}\in C_{k}$, we have
    \[
        \emptyset\neq\set{P\in\mathbf{S}}{\rho_{C_{k}}\underleq{T_{c}\downarrow_{k-1}}P\underleqtop{T_{c}\downarrow_{k-1}}\sigma_{C_{k}}}\subseteq\set{P\in\mathbf{S}}{P\underleqtop{T_{c}\downarrow_{k-1}}\sigma_{C_{k}}}.
    \]
    On the other hand, for each $k\in\{2,\ldots,n-1\}$, let $\rho_{k}$
    denote the restriction of $\rho$ to $\Sigma_{k}$. Then, for all
    $c:T_{c}\downarrow_{k-1}\in C_{k}$, by observing that for all $a,b\in\{0,\ldots,n-1\},$
    we have $c_{b}^{a}\downarrow_{k-1}\,|_{\rho_{k}}\leq c_{b}^{a}|_{\rho}$,
    and hence obtain
    \[
        \emptyset\neq\set{P\in\mathbf{S}}{\rho\underleq{c:T_{c}}P\underleqtop{c:T_{c}}\sigma}\subseteq\set{P\in\mathbf{S}}{\rho_{k}\underleq{c:T_{c}\downarrow_{k-1}}P}.
    \]
    As, for each $k\in\{2,\ldots,n-1\}$, both the sets
    \[
        \set{P\in\mathbf{S}}{P\underleqtop{T_{c}\downarrow_{k-1}}\sigma_{C_{k}}}\qquad\text{and}\qquad\set{P\in\mathbf{S}}{\rho_{k}\underleq{c:T_{c}\downarrow_{k-1}}P}.
    \]
    are non-empty, and because $\rho_{k}\underleq{c:T_{c}\downarrow_{k-1}}P$
    implies $\rho_{k}\underleqtop{T_{c}\downarrow_{k-1}}P$, we apply
    The Bootstrapping Lemma (Lemma~\ref{lem:bootstrapping-lemma2-1})
    to obtain, for each $k\in\{2,\ldots,n-1\}$, that
    \[
        \frac{\sigma_{C_{k}}(n-2)}{\sigma_{C_{k}}(n-1)}\leq\frac{\rho_{k}(n-2)}{\rho_{k}(n-1)}.
    \]
    Therefore, for each $k\in\{2,\ldots,n-1\}$,
    \[
        K_{k-2}=\min\set{\frac{\sigma_{D}(k-2)}{\sigma_{D}(k-1)}}{D\in\essential_{d,k}(\mathbf{S})}\leq\frac{\sigma_{C_{k}}(k-2)}{\sigma_{C_{k}}(k-1)}\leq\frac{\rho_{k}(k-2)}{\rho_{k}(k-1)}=\frac{\rho(k-2)}{\rho(k-1)}.
    \]

    Now, from
    \[
        0<K_{0}\leq\frac{\rho(0)}{\rho(1)},\quad0<K_{1}\leq\frac{\rho(1)}{\rho(2)},\quad\ldots,\quad0<K_{n-3}\leq\frac{\rho(n-3)}{\rho(n-2)},
    \]
    we obtain
    \[
        0<\parenth{\prod_{j=0}^{k-3}K_{j}}\leq\frac{\rho(0)}{\rho(n-2)}.
    \]
    Define
    \[
        \kappa(s):=\begin{cases}
            \parenth{\prod_{j=0}^{k-3}K_{j}} & s=0   \\
            1                                & s=n-2
        \end{cases}\quad(s\in\{0,n-2\}).
    \]
    Then
    \begin{align*}
        0 & <(n-2)_{0}^{0}|_{\kappa}\leq(n-2)_{0}^{0}|_{\rho}\ ,
    \end{align*}
    and since $\rho$ is monotone, for all $a,b,c\in\Sigma_{n}$, with
    $c\leq n-2$, we thus have that $(n-2)_{0}^{0}|_{\kappa}\leq(n-2)_{0}^{0}|_{\rho}\leq c_{b}^{a}|_{\rho}.$
    Hence, for all $c:T_{c}\in C$, we have that no triangulation from
    $\set{P\in\mathbf{S}}{\rho\underleq{c:T_{c}}P\underleqtop{c:T_{c}}\sigma}$
    can have distinct vertices closer than $(n-2)_{0}^{0}|_{\kappa}$
    to each other with respect to the geodesic metric on the unit sphere
    in $\Rd$.

    By compactness of the unit sphere $S\subseteq\Rd$, we set $N\in\N$
    to be the least cardinality that an open cover of $S$ with open geodesic
    balls of radius $(n-2)_{0}^{0}|_{\kappa}/2$ can have. We crucially
    note here that the number $N$ is independent of the choice of the
    element  $C$ from the set $\essential_{d,n}(\mathbf{S})$, and depends
    \emph{only} on the finitely many elements in $\bigcup_{k=2}^{n-1}\essential_{d,k}(\mathbf{S})$.
    Therefore, for every $c:T_{c}\in C$, each triangulation from $\set{P\in\mathbf{S}}{\rho\underleq{c:T_{c}}P\underleqtop{c:T_{c}}\sigma}$
    and hence also $T_{c}$ can have at most $N$ vertices.

    We conclude, for each of the elements $c:T_{c}$ in any $C\in\essential_{d,n}(\mathbf{S})$,
    that the number of vertices that $T_{c}$ can have is at most $N$.
    Hence the cardinality of the set $\essential_{d,n}(\mathbf{S})$ is
    at most the cardinality of the power set of the finite set
    \[
        \set{c:T\in\codes(\{0,\ldots,n-1\})}{\begin{array}{c}
                c\leq n-2, \\
                |\text{vertex set of }T|\leq N
            \end{array}}.\qedhere
    \]
\end{proof}
\begin{cor}
    \label{cor:ess-finite}Let $d,n\in\N$ with $d,n\geq2$. Let $\mathbf{S}$
    be a heteroperturbative set of labeled spherical triangulations of
    unit spheres. The set $\essential_{d,n}(\mathbf{S})$ is finite.
\end{cor}

\begin{proof}
    This is immediate by strong induction using Lemma~\ref{lem:2-ess-finite}
    and Lemma~\ref{lem:essential-sets-finite-induction-step}.
\end{proof}
This allows us to prove our main result, Theorem~\ref{thm:main-theorem},
in the following corollary:
\begin{cor}
    \label{cor:main-result}Let $d,n\in\N$ with $d,n\geq2$. Let $\mathbf{S}$
    be a heteroperturbative set of labeled spherical triangulations of
    unit spheres. The set $\Pi_{d,n}(\mathbf{S})$ is finite.
\end{cor}

\begin{proof}
    This follows from combining Lemma~\ref{lem:Pidn-cardinality-bounded-by-n-essential-sets}
    and Corollary~\ref{cor:ess-finite}.
\end{proof}
As discussed previously, the set $\mathbf{W}$ (cf. Example~\ref{exa:Winter-polytopes-heteroperturbative})
is heteroperturbative, yielding the following result showing that
a fairly interesting subset of $\Pi_{d,n}$ is finite in general:
\begin{cor}
    \label{cor:main-result+winter}Let $d,n\in\N$ with $d,n\geq2$. With
    $\mathbf{W}$, as defined in Example~\ref{exa:Winter-polytopes-heteroperturbative},
    the set $\Pi_{d,n}(\mathbf{W})$ is finite.
\end{cor}

Furthermore, as discussed, $\mathbf{T}_{2}$ (cf. Example~\ref{exa:R2-heteroperturbative})
is also heteroperturbative (in fact $\mathbf{T}_{2}\subseteq\mathbf{W}$),
and since $\Pi_{2,n}=\Pi_{2,n}(\mathbf{T}_{2})$ we regain the previously
known the case for dimension two:
\begin{cor}
    Let $n\in\N$ with $n\geq2$. The set $\Pi_{2,n}$ is finite.
\end{cor}

Given Corollary~\ref{cor:main-result+winter}, one may hope that
the canonical labeled spherical triangulations obtained from compact
packings of any dimension are always elements of $\mathbf{W}$, thereby
resolving Conjecture~\ref{conj:main-conjecture} in full generality.
This seen to be the case for compact disc packings in $\R^{2}$ and
is seen, ex post facto, to be the case for compact sphere packings
in $\R^{3}$ with two sizes of sphere, cf. \cite{FerniqueTwoSpheres2021}.
However this is not true in general. Some of the known compact packings
in $\R^{3}$ with three sizes of sphere have spheres with associated
canonical labeled triangulations of the unit sphere that are not contained
in $\mathbf{W}$. An example is presented in Figure~\ref{fig:fernique-example-not-in-W}
which exhibits the same pathology as described in Figure~\ref{fig:split-meridian}.

\begin{figure}[h]
    \centering \includegraphics[width=0.4\textwidth]{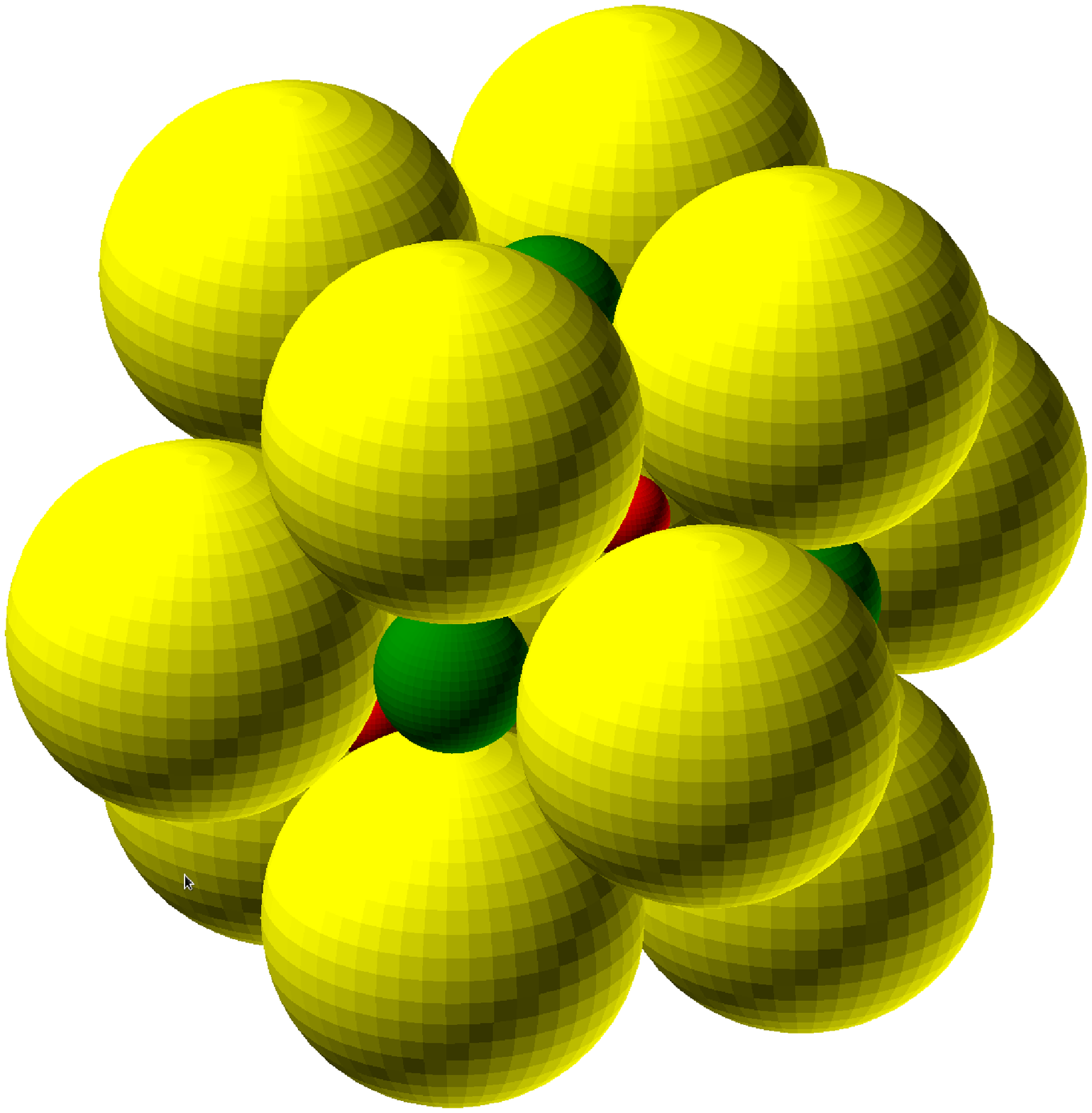}\includegraphics[width=0.4\textwidth]{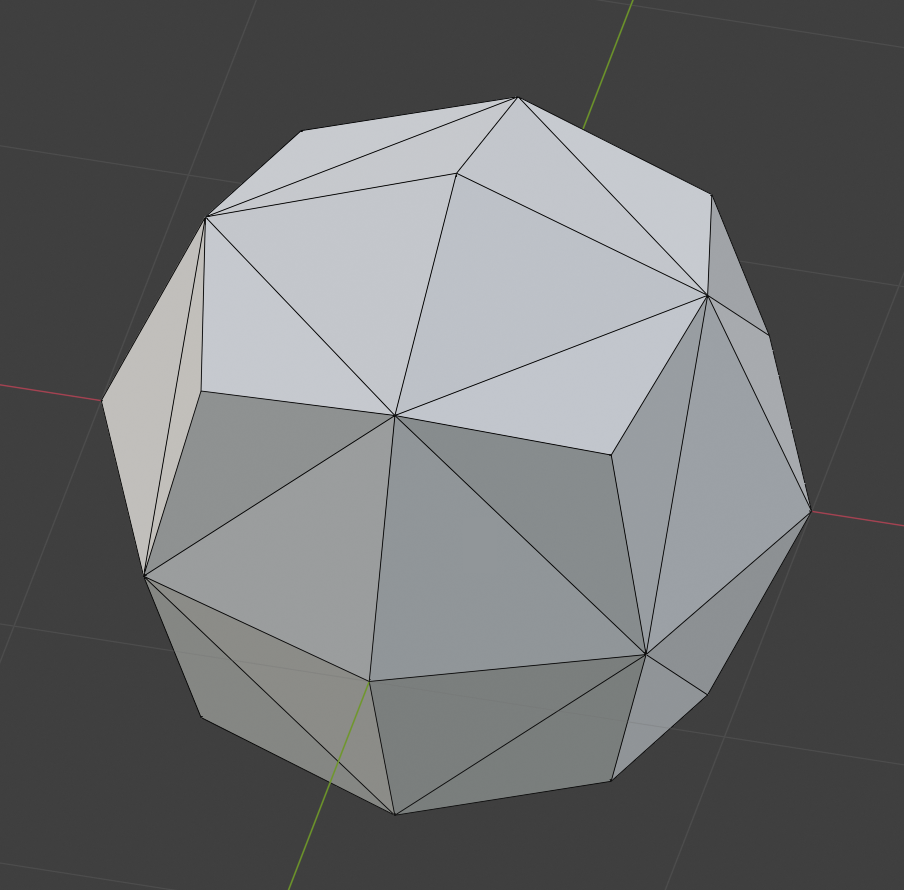}

    \caption{\label{fig:fernique-example-not-in-W}Consider the compact packing
        in $\protect\R^{3}$ from \cite[Theorem~1]{FerniqueThreeSizes} with
        spheres with radii $\sqrt{3/2}-1$ (red), $\sqrt{2}-1$ (green), and
        $1$ (yellow). The canonical labeled spherical triangulation associated
        to the largest spheres in such a packing are not elements of $\mathbf{W}$
        (cf. Example~\ref{exa:Winter-polytopes-heteroperturbative}). On
        a large sphere, the points of contact of the medium and small neighboring
        spheres can ``see'' each other over the chord connecting points of
        contact of other neighboring large spheres. This can be observed on
        the inscribed nonconvex polyhedron in the figure on the right, whose
        central projection onto the circumscribing sphere, coincides with
        the canonical spherical triangulation associated to a large sphere
        from this packing.}
\end{figure}

\begin{acknowledgement*}
    The authors express their gratitude to Martin Winter for numerous
    fruitful discussions and for specifically bringing Example~\ref{exa:R3-not-heteroperturbative}
    to their attention. The authors are furthermore indebted to the anonymous
    referee for their careful reading of the manuscript and for their
    helpful suggestions.
\end{acknowledgement*}
\bibliographystyle{amsalpha}
\bibliography{bibliography}

\newcommand{\etalchar}[1]{$^{#1}$}
\providecommand{\bysame}{\leavevmode\hbox to3em{\hrulefill}\thinspace}
\providecommand{\MR}{\relax\ifhmode\unskip\space\fi MR }
\providecommand{\MRhref}[2]{%
  \href{http://www.ams.org/mathscinet-getitem?mr=#1}{#2}
}
\providecommand{\href}[2]{#2}
\begin{thebibliography}{CCMFT23}

\bibitem[CCMFT23]{ChemistryFernique}
C.~Chinaud-Chaix, N.~Marchenko, T.~Fernique, and S.~Tricard, \emph{Do chemists
  control plane packing{,} i.e. two-dimensional self-assembly{,} at all
  scales?}, New J. Chem. \textbf{47} (2023), 7014--7025.

\bibitem[Cho92]{ChoDihedralAngles}
E.C. Cho, \emph{Dihedral angles of {$n$}-simplices}, Appl. Math. Lett.
  \textbf{5} (1992), no.~4, 55--57. \MR{1341905}

\bibitem[CRS{\etalchar{+}}21]{ChemistryCherniukh}
I.~Cherniukh, G.~Rain{\`o}, T.~St{\"o}ferle, M.~Burian, A.~Travesset,
  D.~Naumenko, H.~Amenitsch, R.~Erni, R.F. Mahrt, M.I. Bodnarchuk, and M.V.
  Kovalenko, \emph{Perovskite-type superlattices from lead halide perovskite
  nanocubes}, Nature \textbf{593} (2021), no.~7860, 535--542.

\bibitem[CS99]{ConwaySloane}
J.~H. Conway and N.~J.~A. Sloane, \emph{Sphere packings, lattices and groups},
  third ed., Springer-Verlag, New York, 1999. \MR{1662447}

\bibitem[Fer19]{FerniqueThreeSizes}
T.~Fernique, \emph{Compact packings of space with three sizes of spheres},
  https://arxiv.org/abs/1912.02293 (2019).

\bibitem[Fer21]{FerniqueTwoSpheres2021}
\bysame, \emph{Compact packings of space with two sizes of spheres}, Discrete
  Comput. Geom. \textbf{65} (2021), no.~4, 1287--1295. \MR{4249904}

\bibitem[Fer23]{fernique2023-five-packing}
T~Fernique, \emph{Packing unequal disks in the euclidean plane},
  https://arxiv.org/abs/2305.12919 (2023).

\bibitem[FHS21]{Fernique2021}
T.~Fernique, A.~Hashemi, and O.~Sizova, \emph{Compact {P}ackings of the {P}lane
  with {T}hree {S}izes of {D}iscs}, Discrete Comput. Geom. \textbf{66} (2021),
  no.~2, 613--635. \MR{4292755}

\bibitem[Ken06]{Kennedy2006}
T.~Kennedy, \emph{Compact packings of the plane with two sizes of discs},
  Discrete Comput. Geom. \textbf{35} (2006), no.~2, 255--267. \MR{2195054}

\bibitem[Lue69]{LuenbergerOptimization}
D.G. Luenberger, \emph{Optimization by vector space methods}, John Wiley \&
  Sons, Inc., New York-London-Sydney, 1969. \MR{0238472}

\bibitem[Mes20]{Messerschmidt2020}
M.~Messerschmidt, \emph{On compact packings of the plane with circles of three
  radii}, Comput. Geom. \textbf{86} (2020).

\bibitem[Mes23]{Messerschmidt2d}
\bysame, \emph{The number of configurations of radii that can occur in compact
  packings of the plane with discs of n sizes is finite}, Discrete Comput.
  Geom. (2023).

\bibitem[Mun84]{Munkres}
J.R. Munkres, \emph{Elements of algebraic topology}, Addison-Wesley Publishing
  Company, Menlo Park, CA, 1984. \MR{755006}

\bibitem[PDKM15]{ChemistryPaik}
T.~Paik, B.T. Diroll, C.R. Kagan, and C.B. Murray, \emph{Binary and ternary
  superlattices self-assembled from colloidal nanodisks and nanorods}, Journal
  of the American Chemical Society \textbf{137} (2015), no.~20, 6662--6669.

\bibitem[Win23]{WinterPolytopes}
M.~Winter, \emph{Rigidity, tensegrity and reconstruction of polytopes under
  metric constraints}, https://arxiv.org/abs/2302.14194 (2023).

\end{thebibliography}

\end{document}